\documentclass[12pt]{article}
\usepackage{amsmath,amssymb, MnSymbol, amsthm}
\usepackage{amsfonts}
\usepackage{graphicx}
\usepackage{color}

\newtheorem{theo}{Theorem}[section]  
\newtheorem{coro}[theo]{Corollary}
\newtheorem{prop}[theo]{Proposition}

\newtheorem{rem}[theo]{Remark}
\newtheorem{exm}[theo]{Example}

\theoremstyle{definition}
\newtheorem{defi}[theo]{Definition}

\begin{document}

\title{First-Order Modal  Logic: \\
Frame Definability and Lindstr\"om Theorems}
\author{R. Zoghifard\footnote{Department of Mathematics and Computer Science, Amirkabir University of Technology, Tehran, Iran. E-mail: r.zoghifard@aut.ac.ir.}\\ M. Pourmahdian\footnote{Department of Mathematics and Computer Science, Amirkabir University of Technology, Tehran, Iran  and School of Mathematics, Institute for Research in Fundamental Sciences (IPM), Tehran, Iran.  E-mail: pourmahd@ipm.ir.}}

\date{}
\maketitle

\begin{abstract}
This paper involves generalizing the Goldblatt-Thomason and the Lindstr\"om characterization theorems to first-order modal logic.
\end{abstract}


\section{Introduction and Preliminaries}
 The purpose of this study  is to extend two well-known theorems of propositional modal logic, namely the Goldblatt-Thomason and the Lindstr\"om characterization theorems, to first-order modal logic. 
 First-order modal logic (FML)  provides  a framework  for incorporating both propositional modal and  classical first-order logics (FOL).  This, from a model theoretic point of view,  means that a  Kripke  frame can be  expanded by a non-empty set  as a domain over which the quantified variables range.  In particular, first-order Kripke models  subsume both propositional Kripke models and classical first-order structures. 

The celebrated Goldblatt-Thomason theorem provides a model theoretic characterization of elementary classes of frames which are definable by a set of propositional modal sentences. This theorem states that an elementary class of Kripke frames is definable by a set of propositional modal formulas if and only if it is closed under bounded morphic images, generated subframes and disjoint unions and reflects ultrafilter extensions.  Our first objective here is to provide a formulation of this theorem in FML.

The second aim of this paper is to study Lindstr\"om type theorems for first-order modal logic. These theorems determine the maximal expressive power of logics in terms of model theoretic concepts.  In a seminal paper \cite{per:elem69}, Lindstr\"om  proved that any abstract logic extending first-order logic with  compactness and the L\"owenheim-Skolem property is not more expressive than the first-order logic. 

To formulate our problems in technical terms, it is worth it to have a quick overview of  basic notions of FML. Suppose $\tau$ is a language consisting of countable numbers of countably many constant and relational symbols. The usual syntactical conventions of first-order logic are assumed here. In particular,  the notions of $\tau$-terms and $\tau$-atomic formulas are defined in the usual way. Furthermore,  first-order modal $\tau$-formulas are defined inductively in the usual pattern as follows:
\[
\phi := P(t_1,\dots ,t_n)\ |\ t_1=t_2\ |\ \neg\phi\ |\ \phi\wedge\phi\ |\ \exists x\phi(x)\ |\ \lozenge\phi,
\]
where $P$ is an $n$-ary predicate and $t_1,\dots, t_n$ are $\tau$-terms. Bound and free variables in a formula are defined as in first-order logic.  A $\tau$-sentence is a formula without any free variable. The other logical connectives ($\vee$, $\rightarrow$, $\leftrightarrow$, $\forall$), and the modal operator $\square$ have their standard definitions. 

A \textit{constant domain Kripke model} is a quadruple $\mathfrak{M}=(W, R, D, I)$, where $W\neq\emptyset$ is a set of possible worlds, $R\subseteq W\times W$ is an accessibility relation, and $D\neq\emptyset$ is a domain. Moreover,
\begin{description}
\item[i.] for each $w\in W$ and any $n$-ary predicate symbol $P$, $I(w,P)\subseteq D^n$,
\item[ii.] for any $w,w'\in W$ and each constant $c$, $I(w,c)=I(w',c)\in D$.
\end{description}
More generally, a \textit{varying domain Kripke model} is a tuple $\mathfrak{M} = (W, R, D, I,\{D(w)\}_{w\in W})$, where $(W,R,D,I)$ is a constant domain model and for each $w\in W$, $D(w)\neq\emptyset$ is a domain of $w$ such that $D=\bigcup_{w\in W} D(w)$. The pair $\mathfrak{F}=(W,R)$ and the triple  $\mathfrak{S}=(W,R,D)$  are respectively called  \textit{frame} and  \textit{skeleton}. A Kripke model $\mathfrak{M}=(W, R, D, I)$ is said to be based on the frame $\mathfrak{F}=(W,R)$ or skeleton $\mathfrak{S}=(W,R,D)$.
For a given first-order Kripke model $\mathfrak{M}$ and  any $w\in W$, the pair $(\mathfrak{M},w)$ is called a pointed  model.

An \textit{assignment} is a function $\sigma$, which assigns to each variable $v$ an element $\sigma(v)$ inside $D$. Two assignments $\sigma$ and $\sigma'$ are $x$-\textit{variant} if $\sigma(y) = \sigma'(y)$ for all variables $y \neq x$. We denote this by $\sigma \sim_x \sigma'$.

The notation $t^{\mathfrak{M},\sigma}$ is used for the interpretation of $t$ in $\mathfrak{M}$ under assignment $\sigma$.

Let $w$ be an arbitrary world of a varying domain model $\mathfrak{M}$ and $\sigma$ be any assignment. For a $\tau$-formula $\phi(\bar{x})$, The notion of ``$\phi(\bar{x})$ is satisfied at $w$ with respect to $\sigma$'' is defined inductively as follows.
\begin{itemize}
\item $\mathfrak{M},w\models_\sigma t_1=t_2$  if and only if  $t_1^{\mathfrak{M},\sigma}=t_2^{\mathfrak{M},\sigma}$.
\item $\mathfrak{M},w\models_\sigma P(t_1,\dots, t_n)$  if and only if  $(t_1^{\mathfrak{M},\sigma},\dots, t_n^{\mathfrak{M},\sigma})\in I(w,P)$, $P$ is an $n$-ary predicate.
\item $\mathfrak{M},w\models_\sigma \neg\phi$ if and only if  $\mathfrak{M},w\not\models_\sigma \phi$.
\item $\mathfrak{M},w\models_\sigma \phi\wedge\psi$ if and only if  $\mathfrak{M},w\models_\sigma \phi$ and $\mathfrak{M},w\models_\sigma \psi$.
\item $\mathfrak{M},w\models_\sigma \exists x\phi(x)$ if and only if there is an $x$-variant $\sigma'$ of $\sigma$ with $\sigma'(x)\in D(w)$ and $\mathfrak{M},w\models_{\sigma'} \phi(x)$.
\item $\mathfrak{M},w\models_\sigma \lozenge\phi$  if and only if  for some $w'\in W$ with $wRw'$, $\mathfrak{M},w'\models_\sigma\phi$.
\end{itemize}

Note that while $I(w,P)$ is not assumed to be a subset of $D(w)$, the existential formulas are  satisfied in world $w$  exactly when a witness can be found within $D(w)$.\footnote{In some literature, e.g \cite{benthem:frame10} and \cite{gabbay:quant09}, the satisfaction of predicates in any world is related to its domain that is $P(a_1,\dots, a_n)$ is true in $w$ if and only if $a_i\in D(w)$ and $(a_1,\dots , a_n)\in I(w,P)$. But we accept the notion which is used in many other references like \cite{fitting:first98} and \cite{braun:first06}. }

Let $\mathfrak{M}$ and $\mathfrak{N}$ be two first-order Kripke models. We call $\mathfrak{M}$ a \textit{submodel} of $\mathfrak{N}$ whenever
\begin{description}
\item[i.] $\mathfrak{F}_\mathfrak{M}$ is a subframe of $\mathfrak{F}_\mathfrak{N}$, that is  $W_\mathfrak{M}\subseteq W_\mathfrak{N}$ and  $R_\mathfrak{M}=R_\mathfrak{N}\upharpoonright W_\mathfrak{M}$.   
\item[ii.] For any $w\in W_\mathfrak{M}$, $D_\mathfrak{M}(w)\subseteq D_\mathfrak{N}(w)$.
\item[iii.] For any $w\in W_\mathfrak{M}$ and any $n$-ary predicate $P$, $I_\mathfrak{M}(w,P)=I_\mathfrak{N}(w,P) \cap D^n_\mathfrak{M}$. Also, $I_\mathfrak{M}(c)=I_\mathfrak{N}(c)$ for all constant symbols $c$.
\end{description}
In this situation we use the notation $\mathfrak{M}\subseteq \mathfrak{N}$.

The model $\mathfrak{M}$ is an \textit{elementary submodel} of $\mathfrak{N}$, denoted by $\mathfrak{M}\preceq \mathfrak{N}$, whenever $\mathfrak{M}\subseteq \mathfrak{N}$ and for any $w\in W_\mathfrak{M}$, $a_1,\dots,a_n\in D_\mathfrak{M}$ and any formula $\phi(x_1,\dots,x_n)$,
$$\mathfrak{M},w\models\phi(a_1,\dots,a_n)\ \text{ if and only if }\ \mathfrak{N},w\models\phi(a_1,\dots,a_n).$$

Now the first part of  this paper is devoted to considering the question of  which elementary classes of frames are definable by a set of first-order modal sentences. This is  a natural generalization of the Goldblatt-Thomason theorem  which was also studied by many authors for various extensions of propositional modal logic. (For example \cite{cate:model05} for hybrid logic, \cite{sano:gtgraded10} for graded modal logic, and for some other extensions \cite{let:generalgt08}.)

Here we answer this question fully by the following theorem (Theorem \ref{gt}).

\textbf{Theorem.} 
\textit{Let $K$ be an elementary class of frames. Then $K$ is  definable  by a set of first-order modal sentences if and only if it is closed under bounded morphic images, generated subframes,  and disjoint unions.}

As we have mentioned the second target of our research is to maintain some versions of Lindstr\"om theorem for first-order modal logic. 
The original theorem of Lindstr\"om has been also extended to propositional modal logic.  De Rijke in \cite{rijke:extend93, rijke:lind95} showed that any abstract logic extending modal logic with the finite depth property is equivalent to modal logic. 
 Van Benthem in \cite{benthem:lind07} showed that the finite depth property is captured by compactness, relativization and bisimulation invariance,  improving de Rijke's result. 
  This result shows that the propositional modal logic is the strongest logic satisfying compactness and the relativization property and is invariant under bisimulation.  Later, van Benthem, ten Cate and V\"a\"an\"anen in \cite{benthem:lind09} proved a Lindstr\"om theorem for some fragments of first-order logic.   Since  the finite depth property fails for modal logic with global modality and guarded fragment of first-order logic, Otto and Piro in \cite{otto:lind08} proved the Lindstr\"om theorem for these logics by using compactness, corresponding bisimulation invariance and the Tarski union property.   By generalizing this method, Enqvist in \cite{enq:general13}  proved a version of Lindstr\"om's theorem for any normal modal logic corresponding to  a class of Kripke frames definable by a set of strict universal Horn formulas.   Also Venema and Kurz in \cite{venema:coalg10} used  coalgebraic methods for proving a version of Lindstr\"om's theorem. Enqvist in \cite{enq:new14} proved a coalgebraic Lindstr\"om theorem that generalizes the van Benthem's result.

Here we show that first-order modal logic with respect to constant domain Kripke models is a logic which has the  maximal expressive power among the logics satisfying the compactness property and is invariant under bisimulation.

Furthermore, using a different method, based on Tarski union property, we give a Lindstr\"om theorem with respect to class of varying domain Kripke models. 

It is worth mentioning that the key point in proving both results here is to view first-order Kripke models as first-order structures in a suitable language in which the basic properties of these structures can be expressed. Of course the expressive power of FML comparing to propositional modal logic is the key advantage which allows us to prove both the Goldblatt-Thomason and the Lindstr\"om Theorems. We must indicate that this well-known method is specified as modal correspondence theory.

The organization of this paper is as follows. In section \ref{sec model} some basic model theoretic results for FML are proved. Also, the notion of saturated models is reviewed and  a version of the Hennessy-Milner theorem for FML is established. These developments yield us some versions of the Goldblatt-Thomason theorem for FML. 

In section \ref{character}, at first the notion of abstract logic is reviewed.  In subsection \ref{seclind}, by considering only constant domain Kripke models we adapt the original proof of  Lindstr\"om's theorem for first-order logic in an elaborate way to prove the corresponding Lindstr\"om  theorem for first-order modal logic.  
In subsection \ref{other lind} a version of Lindstr\"om theorem is given with respect to varying domain Kripke models based on the Tarski union property method.

We will assume that the reader is familiar with the basic notions of propositional modal logic. Also, for further reading on first-order modal logic we refer to \cite{cress:modal01, fitting:first98, braun:first06}.

\section{Model Theory of First-Order Modal Logic}\label{sec model}

In this section some model theoretic  notions for FML are introduced. These methods, in particular allow us  to extend a version  of the Goldblatt-Thomason theorem to first-order modal logic.

\subsection{Bisimulation and Invariance Theorems}\label{sec bisim}

In this subsection the key notions  of bisimulation and saturation are discussed for first-order Kripke models and some of their model theoretic properties are explored. These central concepts play an important role in proving the Lindstr\"om theorem for FML.   

The bisimulation notion is obtained by combining the usual definition of bisimulation for the propositional modal logic with the corresponding  notion of potential isomorphism from first-order logic. 
See Chapter 11 of \cite{benthem:open10}, in which a similar notion of world-object bisimulation for certain varying domain models, called cumulation domain, 
 is given.

Let $D^\star$ be the set of all finite sequences over $D$. For $w\in W$ and $a_1\dots a_n\in D^\star$ we use the abbreviation $w\bar{a}$ instead of $(w,a_1\dots a_n)$. 

All Kripke models in this subsection are assumed to be varying domain.

\begin{defi}\label{bisim}
Let $\mathfrak{M}$  and $\mathfrak{N}$ be two Kripke models. The relation $Z\subseteq (W_\mathfrak{M}\times D_\mathfrak{M}^\star)\times (W_\mathfrak{N}\times D_\mathfrak{N}^\star)$  is called a \textit{bisimulation} if  for any $((w,\bar{a}),(v,\bar{b}))\in Z$ with $|\bar{a}|=|\bar{b}|$  the following three conditions hold:
\begin{enumerate}
\item for any atomic formula $\phi(\bar{x})$, we have $\mathfrak{M},w\models \phi(\bar{a})$ if and only if $\mathfrak{N},v\models \phi(\bar{b})$.
\item \begin{description}
\item ($\lozenge$-forth) For all $w'\in W_\mathfrak{M}$ with $wR_\mathfrak{M}w'$ there is $v'\in W_\mathfrak{N}$ such that $vR_\mathfrak{N}v'$ and $(w',\bar{a}) Z (v',\bar{b})$.
\item ($\lozenge$-back) For all $v'\in W_\mathfrak{N}$ with $vR_\mathfrak{N}v'$ there is $w'\in W_\mathfrak{M}$ such that $wR_\mathfrak{M}w'$ and $(w',\bar{a}) Z (v',\bar{b})$.
\end{description}  
\item \begin{description}
\item ($\exists$-forth) For any $c\in D_\mathfrak{M}(w)$ there is $d\in D_\mathfrak{N}(v)$ with $(w,\bar{a}c) Z (v,\bar{b}d)$.
\item ($\exists$-back) For any $d\in D_\mathfrak{N}(v)$ there is $c\in D_\mathfrak{M}(w)$ with $(w,\bar{a}c) Z (v,\bar{b}d)$.
\end{description}
\end{enumerate}
\end{defi}
Also call $w\bar{a}$ and $v\bar{b}$  bisimilar if there is a bisimulation $Z$ between $\mathfrak{M}$  and $\mathfrak{N}$ such that $(w,\bar{a}) Z (v,\bar{b})$. In this situation we use the notation $(\mathfrak{M},w\bar{a})\rightleftarrows(\mathfrak{N},v\bar{b})$. By $(w,v)\in Z$ we mean that $(w\lambda,v\lambda)\in Z$ where $\lambda$ is the empty sequence in $D^\star$.
 
Note that the notion of bisimulation generalizes both the notion of bisimulation for propositional modal logic and potential isomorphism for first-order logic. 
The following proposition can then be proved by induction on the complexity of formulas.

\begin{prop}\label{bis}
Suppose  $(\mathfrak{M},w\bar{a})\rightleftarrows(\mathfrak{N},v\bar{b})$. Then for any formula $\phi(\bar{x})$ with $|\bar{x}|=|\bar{a}|=|\bar{b}|$, 
$$\mathfrak{M},w\models\phi(\bar{a})\ \text{ if and only if }\ \mathfrak{N},v\models\phi(\bar{b}).$$
\end{prop}

One of the key properties in first-order logic  is that for a given finite language $\tau$, two first-order $\tau$-structures $\mathfrak{M}$ and $\mathfrak{N}$ are elementary equivalent if and only if they are $k$-partially isomorphic  for all $k \in \mathbb{N}$. One of the advantages of the above definition of bisimilarity  is that it helps generalizing this property of first-order logic to the  present context. To this end, for every pair  of natural numbers $l,r$ the concept of $(l,r)$-bisimilarity is defined.

Let $\leq$ be a componentwise order on $\mathbb{N}\times\mathbb{N}$; that is, $(l,r)\leq (l',r')$ if and only if $l\leq l'$ and $r\leq r'$. 

\begin{defi}[$(l,r)$-Bisimulation]
Let $(\mathfrak{M},w)$ and $(\mathfrak{N},v)$ be two pointed Kripke models and  $\bar{a}\in D_\mathfrak{M}^\star$ and $\bar{b}\in D_\mathfrak{N}^\star$ be sequences of the same length. We say that  $w\bar{a}$ and $v\bar{b}$ with $|\bar{a}|=|\bar{b}|$ are $(l,r)$-bisimilar, and denote it by $(\mathfrak{M},w\bar{a}) \rightleftarrows_{(l,r)} (\mathfrak{N},v\bar{b})$, if there exists a sequence of relations $\langle Z_{(i,j)}: 0\leq i\leq l, 0\leq j\leq r\rangle$ such that:
\begin{description}
\item (i) $w\bar{a}Z_{(l,r)} v\bar{b}$.
\item (ii) If $w'\bar{c} Z_{(0,0)} v'\bar{d}$ then  $|\bar{c}|=|\bar{d}|$ and $\mathfrak{M},w'\models \phi(\bar{c})$ if and only if $\mathfrak{N},v'\models \phi(\bar{d})$,  for any atomic formula $\phi(\bar{x})$.
\item (iii) $\lozenge$-forth and $\lozenge$-back: if $w'\bar{c} Z_{(i+1,j)} v'\bar{d}$ and $w' R_\mathfrak{M} w''$, then there is $v''\in W_\mathfrak{N}$ with $v' R_\mathfrak{N} v''$ such that $w''\bar{c} Z_{(i,j)} v''\bar{d}$ and vice versa.
\item  (iv) $\exists$-forth and $\exists$-back: if $w'\bar{c} Z_{(i,j+1)} v'\bar{d}$ and $c'\in D_\mathfrak{M}(w')$, then there is $d'\in D_\mathfrak{N}(v')$ with $w'\bar{c}c' Z_{(i,j)} v'\bar{d}d'$ and vice versa.
\end{description} 
\end{defi}

Regarding the definition above, in the following we define degree and a function that assigns to each formula a pair of natural numbers. 

\begin{defi}
For the first-order modal formula $\phi(\bar{x})$ the \textit{degree} of  $\phi$ is inductively defined as follows.
\begin{itemize}
\item $\delta(\phi(\bar{x}))=(0,0)$, for every atomic formula $\phi$.
\item $\delta(\neg\phi)=\delta(\phi)$.
\item If $\delta(\phi)=(l,r)$ and $\delta(\psi)=(l',r')$, then 
 $\delta(\phi\wedge\psi) = (\text{Max}(l,l') , \text{Max}(r,r'))$.
\item If $\delta(\phi)=(l,r)$, then $\delta(\lozenge\phi)=(l+1,r)$.
\item If $\delta(\phi)=(l,r)$, then $\delta(\exists x\phi(x))=(l,r+1)$.
\end{itemize}
\end{defi}

Assume that $\Gamma$ is any set of first-order modal formulas with at most $n$ free variables among $\bar{x}=x_1,\dots,x_n$. Let $(\mathfrak{M},w)$ and $(\mathfrak{N},v)$  be two pointed  Kripke models with $w\in W_\mathfrak{M}$, $\bar{a}\in D_\mathfrak{M}$, $v\in W_\mathfrak{N}$  and $\bar{b}\in D_\mathfrak{N}$ where $|\bar{a}|=|\bar{b}|=n$. The notation  $(\mathfrak{M},w\bar{a})\equiv_\Gamma (\mathfrak{N},v\bar{b})$ means that for any $\phi(\bar{x})\in\Gamma$, $\mathfrak{M},w\models\phi(\bar{a})$ if and only if $\mathfrak{N},v\models\phi(\bar{b})$. 
If $\Gamma$ is a set of first-order modal formulas all with at most $n$ free variables among $\bar{x}$, we omit the subscript $\Gamma$ and  write $(\mathfrak{M},w\bar{a})\equiv(\mathfrak{N},v\bar{b})$. 
For $\Gamma=\{\phi(\bar{x})\ | \delta(\phi)\leq(l,r)\}$ the notation above is $(\mathfrak{M},w\bar{a})\equiv_{(l,r)} (\mathfrak{N},v\bar{b})$.
Finally, two pointed Kripke models $(\mathfrak{M},w)$ and $(\mathfrak{N},v)$ are \textit{elementary equivalent}, $(\mathfrak{M},w)\equiv(\mathfrak{N},v)$, if for any sentence $\phi$, we have $\mathfrak{M},w\models \phi$ if and only if $\mathfrak{N},v\models \phi$. 

\begin{prop}\label{fbisim}
Let $(\mathfrak{M},w)$ and $(\mathfrak{N},v)$ be two pointed Kripke models. If $(\mathfrak{M},w\bar{a})\rightleftarrows_{(l,r)}(\mathfrak{N},v\bar{b})$, then $(\mathfrak{M},w\bar{a})\equiv_{(l,r)}(\mathfrak{N},v\bar{b})$.
\end{prop}
\begin{proof}
By induction on $(l,r)$.
\end{proof}

\begin{prop}\label{ckf}
Suppose $\tau$ is finite. Then $(\mathfrak{M},w\bar{a})\rightleftarrows_{(l,r)} (\mathfrak{N},v\bar{b})$ if and only if $(\mathfrak{M},w\bar{a})\equiv_{(l,r)} (\mathfrak{N},v\bar{b})$.
\end{prop}

\begin{proof}
The left to right direction follows from Proposition \ref{fbisim}.

To show the other direction, note that since the language $\tau$ is finite,  by induction on $(l,r)$, up to logical equivalence, there is a finite set of formulas with degree at most $(l,r)$.

So,  it suffices to prove that the sequence $\langle Z_{(i,j)}:  1\leq i\leq l, 1\leq j\leq r\rangle$, where 
$$
Z_{(i,j)}=\{ (w\bar{a}c_1\dots c_{r-j},v\bar{b}d_1\dots d_{r-j})\ |\ w\bar{a}c_1\dots c_{r-j}\equiv_{(i,j)} v\bar{b}d_1\dots d_{r-j}\}
$$
is a $(l,r)$ bisimulation. But this can be proved in a similar way as in Proposition 2.31 in \cite{blac:modal01}.
\end{proof}

Fixing a bijective coding function $\langle,\rangle: \mathbb{N}\times\mathbb{N}\rightarrow\mathbb{N}$ we say that $(\mathfrak{M},w)$ and $(\mathfrak{N},v)$ are $k$-bisimilar provided that if $\langle l,r\rangle=k$.
Since the coding function $\langle,\rangle$ is bijective, there are unique functions $\textit{left}$ and $\textit{right}$ from $\mathbb{N}$ to $\mathbb{N}$ such that $\textit{left}(k)=l$ and $\textit{right}(k)=r$ if and only if $k=\langle l,r\rangle$. 
Using this coding function we may adapt Definition 2.4 to say that  two models are  $k$-bisimilar for $k\in N$.  

First-order modal logic can be viewed as a fragment of two-sorted first-order logic. This is beneficial especially  in proving  number of properties for this logic, for example compactness and the L\"owenheim-Skolem  property.   To this end, for a given first-order modal language $\tau$ one can consider a two-sorted language $\tau^{cor}$ whose sorts are denoted by $s_W$ and $s_O$. The first sort $s_W$ corresponds to the set of possible worlds and the second sort $s_O$ distinguishes  the set of objects.  So $\tau^{cor}$ has two binary predicates  $R(u,u')$ and $E(u,x)$  where $u,u'\in s_W$ and $x\in s_O$. The intuitive meaning of  $R(u,u')$ is that $u$ is related to $u'$ via the accessibility relation $R$ and $E(u,x)$ is interpreted as the object $x$ being in the domain $D(u)$.  
The language $\tau^{cor}$ includes all constant symbols of $\tau$ and  for a given $n$-ary predicate $p(x_1,\dots ,x_n)\in\tau$, $\tau^{cor}$ includes $(n+1)$-ary predicate $P(u,x_1,\dots , x_n)$.  

So, any first-order modal $\tau$-formula can be inductively translated to a $\tau^{cor}$-formula as follows:
\begin{itemize}
\item $ST_u(t_1=t_2) = t_1=t_2$, 
\item $ST_u(P(x_1,\dots,x_n)) = P(u,x_1,\dots,x_n)$,
\item $ST_u(\neg\phi)=\neg ST_u(\phi)$,
\item $ST_u(\phi\wedge\psi) = ST_u(\phi) \wedge ST_u(\psi)$,
\item $ST_u(\lozenge\phi) = \exists u'(R(u,u')\wedge ST_{u'}(\phi))$,
\item $ST_u(\exists x\phi(x)) = \exists x (E(u,x)\wedge ST_u(\phi))$.
\end{itemize}

Furthermore, any first-order Kripke model $\mathfrak{M}$ can be viewed as a $\tau^{cor}$-structure $\mathfrak{M}^\ast$ in the natural way.

The following proposition is obtained by induction on the complexity of first-order modal formulas.

\begin{prop}
Let $\phi(x_1,\dots,x_n)$ be any first-order modal formula. Then for any  Kripke model $\mathfrak{M}$ with $w\in W$ and $a_1,\dots a_n \in D$
\[
\mathfrak{M},w\models\phi(a_1,\dots , a_n) \ \text{ if and only if  } \ \mathfrak{M}^\ast\models ST_u(\phi)(w,a_1,\dots ,a_n).
\]
\end{prop}

Similar to propositional modal logic the van Benthem invariance theorem also holds for FML.

\begin{theo}[Invariance Theorem \cite{benthem:open10}]\label{invar}
Let $\alpha$ be a first-order $\tau^{cor}$-formula. Then $\alpha$  is a translation of some modal $\tau$-formula if and only if it  is invariant under bisimulations.
\end{theo}

The notion of a modally-saturated model, introduced below,  is a modification of $\omega$-saturation in first-order model theory. To this end, we first need to define the notion of a type over a Kripke model.
For a given varying domain Kripke model $\mathfrak{M}=(W, R, D, I)$ 
 and a finite subset $A\subseteq D$, the language $\tau_A$ is an expansion of $\tau$ by adding some new constant symbols $c_a$ for all $a\in A$. The  $\tau_A$-Kripke model $\mathfrak{M}_A$ expands $\mathfrak{M}$ naturally by interpreting any constant symbol $c_a$ by $a$ itself.

\begin{defi}\label{type}
Let $\Gamma(\bar{x})$ be a set of $\tau_A$-formulas whose free variables are among $\bar{x}$. A set of formulas $\Gamma(\bar{x})$ is an \textit{$\exists$-type} of $(\mathfrak{M}_A,w)$ if for all finite subsets $\Gamma_0(\bar{x})$ of $\Gamma(\bar{x})$, we have $\mathfrak{M}_A,w\models\exists \bar{x}\bigwedge \Gamma_0(\bar{x})$. 
Similarly, $\Gamma(\bar{x})$ is a \textit{$\lozenge$-type} of $(\mathfrak{M}_A,w)$ with respect to some $\bar{a}\in D$, if $\mathfrak{M}_A,w\models\lozenge\bigwedge\Gamma_0(\bar{a})$ for all finite $\Gamma_0\subseteq\Gamma$.
\end{defi}

A type of $(\mathfrak{M},w)$ is either a $\lozenge$-type or an $\exists$-type of $(\mathfrak{M}_A,w)$ for some finite subset $A\subseteq D$.

\begin{defi}\label{saturat}
An $\exists$-type $\Gamma(\bar{x})$ is realized in $(\mathfrak{M},w)$ if there are $\bar{a}\in D(w)$ such that $\mathfrak{M},w\models\Gamma(\bar{a})$.
Likewise, a $\lozenge$-type $\Gamma(\bar{x})$ is realized in $(\mathfrak{M},w)$ with respect to $\bar{a}$, if there is an element $w'\in W$ such that $wRw'$ and $\mathfrak{M},w'\models\Gamma(\bar{a})$. 

A model $\mathfrak{M}$ is \textit{modally-saturated} (or m-saturated for short) if for every $w\in W$ and each finite subset $A$ of $D$, every type of  $(\mathfrak{M},w)$ is realized in $(\mathfrak{M},w)$.
\end{defi}

Next we present the first-order version of the Hennessy-Milner theorem.

\begin{theo}[Hennessy-Milner Theorem]
Let $\mathfrak{M}$ and $\mathfrak{N}$ be two modally-saturated Kripke models. Then $(\mathfrak{M},w_0\bar{a}_0)\equiv(\mathfrak{N},v_0\bar{b}_0)$ if and only if $(\mathfrak{M},w_0\bar{a}_0)\rightleftarrows(\mathfrak{N},v_0\bar{b}_0)$.
\end{theo}
\begin{proof}
The right to left direction is derived from Proposition \ref{bis}. For the other direction, one can show that the set
 $$Z=\{(w\bar{a},v\bar{b})\ |\ (\mathfrak{M},w\bar{a})\equiv(\mathfrak{N},v\bar{b})\}$$
  is a bisimulation between $\mathfrak{M}$ and $\mathfrak{N}$ with $w_0\bar{a}_0 Z v_0\bar{b}_0$. 
\end{proof}

Let $\mathfrak{M}$ be a Kripke model. We call $\mathfrak{M}$ an $\omega$-saturated model if $\mathfrak{M}^\ast$ is $\omega$-saturated as a  first-order $\tau^{cor}$-structure.  Obviously any $\omega$-saturated model is also m-saturated.

\subsection{Frame Definability}\label{fr def}

Our objective in this section is to prove some form of the Goldblatt-Thomason theorem for first-order modal logic.   The Goldblatt-Thomason theorem states that an elementary class of frames is definable by a set of propositional modal formulas if and only if it is closed under bounded morphic images, generated subframes and disjoint unions, and reflects ultrafilter extensions \cite{blac:modal01}.  Here, we consider the question of which elementary classes of frames are definable by a set of first-order modal sentences. To motivate this question we first give Example \ref{example} which shows that there exist a class of Kripke frames which is not definable by a set of propositional modal formulas but  is definable by a set of first-order modal sentences. While the validity of first-order modal formulas are preserved under  bounded morphic images, generated subframes, and disjoint unions, we see that the validity of first-order modal formulas in frames is no longer reflected by ultrafilter extensions. Furthermore, our main result (Theorem \ref{gt}) shows that if the class of Kripke frames is closed under these three operations, then it can be defined by a set of first-order modal sentences.

Throughout this subsection we focus on constant domain Kripke models. Henceforth, we suppose $\tau$ is a first-order modal language and $\mathfrak{F}$ is a frame. We say that a $\tau$-formula $\phi(x_1,\dots,x_n)$  is valid at a world $w$ of the frame $\mathfrak{F}$ if for any constant domain $\tau$-model $\mathfrak{M}$ based on $\mathfrak{F}$, we have $\mathfrak{M},w\models\forall x_1\dots\forall x_n\phi(x_1,\dots,x_n)$. In this case we write $\mathfrak{F},w\models\phi(x_1,\dots,x_n)$. 
We also say $\phi(x_1,\dots,x_n)$ is valid on the frame $\mathfrak{F}$, denoted by $\mathfrak{F}\models\phi(x_1,\dots,x_n)$, whenever it is valid in any world $w\in W$.

A class of Kripke frames $K$ is \textit{FML-definable} if for some language $\tau$ there exists a set of first-order modal $\tau$-sentences $\Lambda$ such that for any frame $\mathfrak{F}$, we have $\mathfrak{F}\in K$ if and only if all sentences of $\Lambda$ are valid on $\mathfrak{F}$. More generally, a  set of first-order modal formulas $\Lambda$ defines $K$ if, for any Kripke frame  $\mathfrak{F}$ we have $\mathfrak{F}\in K$ if and only if $\mathfrak{F}\models\forall \bar{x}\phi(\bar{x})$ for any formula $\phi(\bar{x})\in\Lambda$.

By the substitution property  one can see that if a class of frames $K$ is definable by a set of propositional modal formulas (PML-definable for short) then it is also FML-definable.

Suppose $\theta(p_1,\dots,p_n)$ is a propositional modal formula and $P_1(x), \dots, P_n(x)$  are atomic formulas where $P_i$, $1\leq i\leq n$, is a unary predicate and $x$ is a single variable.  The substitution of $\theta(p_1,\dots,p_n)$  is  a first-order modal formula $\theta(P_1(x),\dots,P_n(x))$ which is obtained by uniformly replacing each proposition $p_i$ by an atomic formula  $P_i(x)$ for some variable $x$. 

\begin{prop}\label{propdefin}
Let $K$ be any class of frames. If $K$ is PML-definable then it is also FML-definable.
\end{prop}
\begin{proof}
Suppose $K$ is definable by a set of propositional modal formulas $\Lambda_P$. Let $\tau=\{P_1, P_2,\dots,P_n, \dots\}$ consist of countably many unary predicates and put  $\Lambda_{FO}$ to be a set of all formulas of $\forall x\  \theta(P_1(x),\dots, P_n(x))$ for $\theta\in\Lambda_P$. Then one can easily see that $\Lambda_{FO}$ defines $K$.
\end{proof}

The converse of the above proposition is not true in general. This means that there are some classes of frames which are first-order modally definable but not PML-definable.

\begin{exm}\label{example}\emph{
 Consider the class of frames which satisfy the condition that every world has a reflexive accessible  world (i.e. $\forall x\exists y(Rxy\wedge Ryy)$).  As is mentioned in \cite{blac:modal01} page 142,  this class is not definable by any set of propositional modal formulas, since it does not reflect the ultrafilter extension. However this class is definable by the formula $\lozenge \forall x (\square P(x) \rightarrow P(x))$. }

\emph{Clearly, $\lozenge \forall x (\square P(x) \rightarrow P(x))$ is valid on any such frame. For the converse  suppose that $\mathfrak{F}$ is a frame with some world $w_0$ that does not have any reflexive successor. Let  $\mathfrak{M}=(\mathfrak{F},D,I)$ be a constant domain model based on $\mathfrak{F}$ such that $|D|\geq |\{w'\ |\ wRw'\}|$. Moreover, for each $w'$, accessible by $w_0$ there exists a distinct element $d_{w'}\in D$ such that  $P(d_{w'})$ is false in $w'$ but it is true in all successors of $w'$.  Then for any $w'$ with $w_0Rw'$, we have $\mathfrak{M},w'\not\models \forall x (\square P(x) \rightarrow P(x))$. So $\mathfrak{M},w_0\not\models \lozenge \forall x (\square P(x) \rightarrow P(x))$. 
}\end{exm}

The above example shows in particular that the validity of first-order modal sentences is not reflected by ultrafilter extension.  So by this example it is natural to ask if any version of the Goldblatt-Thomason theorem holds for FML. 
 
 Next we prove that the validity of first-order modal formulas is preserved under bounded morphic images, generated subframes and disjoint unions. 
Let us begin by a quick review of the corresponding concepts.

Let $\mathfrak{F}$ and $\mathfrak{G}$ be two Kripke frames. A function $f:W_\mathfrak{F}\rightarrow W_\mathfrak{G}$ is a bounded morphism from $\mathfrak{F}$ to $\mathfrak{G}$ if
\begin{itemize}
\item  $wR_\mathfrak{F} w'$ implies $f(w) R_\mathfrak{G} f(w')$, and
\item if $f(w)R_\mathfrak{G}v$, then there exists $w'\in W_\mathfrak{F}$ such that $f(w')=v$ and  $wR_\mathfrak{F}w'$.
\end{itemize}
If, in addition,   $f$ is surjective, then $\mathfrak{G}$ is called a \textit{bounded morphic image} of $\mathfrak{F}$ and is denoted by $\mathfrak{F} \twoheadrightarrow \mathfrak{G}$.

A frame $\mathfrak{F}$ is a \textit{generated subframe} of $\mathfrak{G}$, $\mathfrak{F}\rightarrowtail\mathfrak{G}$, if $W_\mathfrak{F}\subseteq W_\mathfrak{G}$   and  $R_\mathfrak{F}=R_\mathfrak{G}\upharpoonright W_\mathfrak{F}$ and in addition, for any $w\in W_\mathfrak{F}$ if $wR_\mathfrak{G}w'$ then $w'\in W_\mathfrak{F}$. 

Suppose $\langle \mathfrak{F}_i : i\in I\rangle$ is a family of disjoint Kripke frames. The \textit{disjoint union} of $\mathfrak{F}_i$s is a frame $\mathfrak{F}$ in which $W= \biguplus_{i\in I} W_i$, $R= \biguplus_{i\in I} R_i$. In this case we write $\mathfrak{F}=\biguplus_{i\in I} \mathfrak{F}_i$.

For a Kripke frame $\mathfrak{F}=(W,R)$ let $Uf(W)$  be the set of all ultrafilters over $W$. Define the binary relation $R^{ue}$ on $Uf(W)$ as follows. For any $u,u'\in Uf(W)$, $uR^{ue}u'$ if and only if $m_\lozenge(X)=\{w\ | \ \exists w'\in W\   s.t\  wRw'\}\in u$ for any $X\in u'$. The Kripke frame $ue(\mathfrak{F})=(Uf(W),R^{ue})$ is called an ultrafilter extension of $\mathfrak{F}$. Finally,  for any $w\in W$ set $\hat{w}$ as the principal ultrafilter generated by $w$.

\begin{prop}\label{frmpreserv}
The validity of first-order modal sentences are preserved under bounded morphic images, generated subframes, and disjoint unions.
\end{prop}
\begin{proof}
The proof is a simple generalization of the proof of Theorem 3.14 in \cite{blac:modal01} for propositional version.
\end{proof}

Similar to propositional case one may prove that any frame is a bounded morphic image of the disjoint union of its point-generated subframes.

A first-order modal formula $\phi(x_1,\dots,x_n)$ is called \textit{universal} if it is in the form $\forall y_1,\dots,\forall y_m \psi(x_1,\dots,x_n,y_1,\dots,y_m)$ where $\psi$ is a quantifier-free formula.

\begin{prop}\label{unifram}
Let $K$ be a class of frames definable by a set of universal modal sentences in some language $\tau$. Then $K$ is definable by a set of propositional modal formulas.
\end{prop}
\begin{proof}
Suppose that the class $K$ is defined by a set of universal modal formulas $\Lambda$. By Proposition \ref{frmpreserv} $K$ is closed under bounded morphic images, generated subframes, and disjoint unions.  We show that $K$ reflects ultrafilter extensions, and hence by the Goldblatt-Thomason theorem it is definable by a set of propositional modal formulas.

Assume that $ue(\mathfrak{F})\in K$ for some frame $\mathfrak{F}$. It is enough to show that $\mathfrak{F}\models\Lambda$. Otherwise for some universal sentence $\phi=\forall \bar{x}\psi(\bar{x})\in \Lambda$ there is a model $\mathfrak{M}$ based on $\mathfrak{F}$ and world $w\in W_\mathfrak{F}$ where $\mathfrak{M},w\not\models\phi$. Let $ue(\mathfrak{M})$ be a model based on $ue(\mathfrak{F})$ with the set $D_\mathfrak{M}$ as domain. Also for any $u\in Uf(W)$ and any $n$-ary predicate $P$, $(a_1,\dots,a_n)\in I^{ue}(u,P)$ if and only if $\{w\in W_\mathfrak{F}\ |\  (a_1,\dots,a_n)\in I^{\mathfrak{M}}(w,P)\}\in u$.  Then by induction on the complexity of formulas one may show that for any quantifier-free formula $\psi$, we have $\mathfrak{M},w\models \psi(a_1,\dots,a_n)$ if and only if $ue\mathfrak{M},\hat{w}\models\psi(a_1,\dots,a_n)$. Hence $ue(\mathfrak{M}),\hat{w}\not\models\phi$, a contradiction. 
\end{proof}

The following theorem is a version of the Goldblatt-Thomason theorem for first-order modal logic.  From now on, we assume that $\tau$ is a language containing for each $n\in\mathbb{N}$  a countable set of $n$-ary predicate and a countable set of constant symbols.  

\begin{theo}\label{gt}
Let $K$ be an elementary class of frames. Then $K$ is definable  by a set of first-order modal $\tau$-sentences if and only if it is closed under bounded morphic images, generated subframes,  and disjoint unions.
\end{theo}
\begin{proof}
If $K$ is definable by a set of first-order modal sentences then by Proposition \ref{frmpreserv} it is closed under the three operations as required by the theorem. 

For the other direction, suppose $K$ is closed under bounded morphic images, generated subframes,  and disjoint unions. Let $\Lambda_K$ be the set of all first-order modal $\tau$-sentences valid on any frame $\mathfrak{G}\in K$.  We show that $\Lambda_K$ defines $K$. Clearly, $\Lambda_K$ is valid on any frame of $K$. Conversely, assume $\mathfrak{F}\models\Lambda_K$. Since $K$ is closed under bounded morphic images, generated subframes,  and disjoint unions, without loss of generality we assume that $\mathfrak{F}$ is point-generated by $w_0$.

Let $\tau'= \{R,P\}\cup \{c_w\ | w\in W\}$ be a language consisting of a unary predicate $P$,  a binary predicate $R$,  and new constants $c_w$ for any $w\in W$.  
Let $\mathfrak{M}=(\mathfrak{F},D,I)$ be a model based on $\mathfrak{F}$ such that $D=W$. In each world $w$, put $I(w,P)=\{w\}$.   In all worlds, $R(w,w')$ holds if and only if $(w,w')\in R_\mathfrak{F}$. Finally interpret each constant $c_w$  as $w$.

Take $\Delta$ as the set of all $\tau$-sentences true in $(\mathfrak{M},w)$. Since $\mathfrak{F}$ is point-generated by $w_0$, for each $w\in W$ there is a finite number $n_w$ which is the length of the shortest path from $w_0$ to $w$. So for each $w,w'\in W$ and for all $n\in \mathbb{N}$, $\Delta$ contains the following sentences.
\begin{enumerate}
\item $\lozenge^{n_w} P(c_w)$. \label{cod}
\item $\square^n c_w\neq c_{w'}$, if $w\neq w'$. \label{ineq}
\item $\square^n \exists x (P(x)\wedge \forall y (P(y)\rightarrow x=y))$. \label{uniq}
\item $\square^n R(c_w,c_{w'})$, if $(w,w')\in R$. \label{R(cwcw')}
\item $\square^n \neg R(c_w,c_{w'})$, if $(w,w')\not\in R$.\label{not R(cwcw')}
\item $\square^n \forall x\forall y (P(x)\wedge \lozenge P(y) \rightarrow R(x,y))$. \label{8}
\item $\square^n \forall x\forall y (P(x)\wedge R(x,y)\rightarrow \lozenge P(y))$. \label{9}
\item $\square^n (\exists x (P(x)\wedge \phi(x, c_1,\dots, c_n))\rightarrow \phi(c_\phi,c_1,\dots ,c_n))$,  for any $\phi$ as a $\{R\}$-formula and for some $c_\phi$. \label{10}
\item $\forall x_1,\dots, x_k (\phi(x_1,\dots , x_k) \rightarrow\square^n \phi(x_1,\dots, x_k))$, for any  $\{R\}$-formula $\phi$ . \label{11}
\item $\forall x_1,\dots, x_k (\lozenge^n\phi(x_1,\dots, x_k) \rightarrow \phi(x_1,\dots, x_k))$, for any $\{R\}$-formula $\phi$ . \label{12}
\end{enumerate}

\begin{description}
\item[Claim1.]
$\Delta$ is finitely satisfiable in $K$. 

Otherwise there is  $\delta\in \Delta$ which is not true in any constant domain model based on the frames of $K$. So $\neg\delta \in \Lambda_K$. But this contradicts with the fact that $\mathfrak{F}\models\Lambda_K$.
\end{description}

Hence there is a frame $\mathfrak{G}\in K$ and a constant domain model $(\mathfrak{N},v_0)$ based on $\mathfrak{G}$   such that $\mathfrak{N},v_0\models \Delta$. Once again without loss of generality we assume that $\mathfrak{G}$ is generated by $v_0$.

Now let $\mathfrak{G}'=(W',R')$ where $W'=\{a\in D_\mathfrak{N}\ | \ \mathfrak{N},v\models P(a), \text{ for some } v\in \mathfrak{G}\}$, and $(a,b)\in R'$ if and only if $\mathfrak{N}\models R(a,b)$. 

\begin{description}
\item[Claim 2.] $\mathfrak{F}$ is isomorphic to a first-order elementary submodel of $\mathfrak{G}'$ as a model in the language $\{R\}$.

Define the map $f: W\rightarrow W'$  sending any $w$ to $(c_w)^\mathfrak{N}$, the interpretation of constant $c_w$ in $\mathfrak{N}$. 
By  conditions \ref{cod} and \ref{ineq}  $f$ is an injective function. Furthermore, $f$ is an elementary embedding, since conditions \ref{R(cwcw')}, \ref{10}, and \ref{12} hold.

\item[Claim 3.] $\mathfrak{G}\twoheadrightarrow\mathfrak{G}'$.

Let $h: \mathfrak{G}\rightarrow \mathfrak{G}'$ be a function such that $h(v)=a_v$ where $a_v$ is the unique element of $D_\mathfrak{N}$ such that $P(a_v)$ is true in $v$. 
Since $\mathfrak{G}$ is generated by $v_0$ and $(\mathfrak{N},v_0)$ satisfies condition \ref{uniq} from $\Delta$,  $h$ is a function and is by definition surjective. The forth property is obtained by  conditions \ref{8} and \ref{12}, and the condition \ref{9}  implies the back property. So $\mathfrak{G}'$ is a bounded morphic image of $\mathfrak{G}$. 
\end{description}
By closure of $K$ under bounded morphic images we have $\mathfrak{G}'\in K$. Now since $K$ is elementary, Claim 2  implies that $\mathfrak{F}\in K$.  
\end{proof}

\section{Lindstr\"om Theorem for First-Order Modal Logic}\label{character}

In this section we  present various versions of Lindstr\"om's theorem for first-order modal logic. These theorems can be seen as some model theoretic characterizations of this logic.
 First the notion of an abstract logic is reviewed.

\subsection{Abstract Logic}

Throughout this section all the languages are assumed to be relational. So any language $\sigma=\mathcal{P}\cup\mathcal{C}$ is assumed to be a relational language, where $\mathcal{P}$ and $\mathcal{C}$ are a set of  predicate and constant symbols, respectively. In some cases we use the notation $(P,n)$ to indicate that $P$ is an $n$-ary predicate.

For a relational language $\sigma$, by a $\sigma$-structure (or $\sigma$-model)  we mean a  pointed varying domain Kripke $\sigma$-model. Let $Str[\sigma]$ be the class of all $\sigma$-structures.

For given $\mathfrak{M}\in Str[\sigma]$ and $\sigma'$ an expansion of $\sigma$, 
one may  expand $\mathfrak{M}$ by interpreting  any new symbols of  $\sigma'\diagdown\sigma$ inside $\mathfrak{M}$ and form a $\sigma'$-structure $\mathfrak{M}'$. 
In this situation, $\mathfrak{M}$ is called the reduct of $\mathfrak{M}'$ to the language $\sigma$; this is denoted by $\mathfrak{M}=\mathfrak{M}'\upharpoonright\mathcal{\sigma}$.

\begin{defi}\label{abstract}
An \textit{abstract logic} is a pair $(L,\models_L)$, where
\begin{enumerate}
\item  $L$ is a map that assigns to each language $\sigma$ a class $L[\sigma]$; whose members are called  (abstract) $\sigma$-sentences. Likewise an $L[\sigma]$-theory is a set of $L[\sigma]$-sentences,
\item  $\models_L$ is a satisfaction relation, that is  a binary relation between $\sigma$-structures and $\sigma$-sentences,
 \end{enumerate}
 which satisfies the following properties: 
\begin{itemize}
\item if $\sigma\subseteq \sigma'$, then $L[\sigma]\subseteq L[\sigma']$.
\item (Occurrence)  For each $\phi\in L$ there is a finite language $\sigma_\phi$, such that for any $\sigma$-structure $(\mathfrak{M},w)$, $\mathfrak{M},w\models_L\phi$ is defined if and only if $\sigma_\phi\subseteq\sigma$.
\item (Expansion) For $\sigma\subseteq\sigma'$, if $\sigma'$-structure $(\mathfrak{M}',w)$ is an expansion of $\sigma$-structure $(\mathfrak{M},w)$, then for any $\phi$ with $\sigma_\phi\subseteq \sigma$,  $\mathfrak{M},w\models_L\phi$ implies $\mathfrak{M}',w\models_L\phi$. 
\end{itemize}
\end{defi}

An abstract logic $L'$ extends $L$, and denoted by $L\subseteq L'$, if for any $\sigma$ and any sentence $\phi\in L[\sigma]$, one can find a sentence $\psi\in L'[\sigma]$ such that for a give $\sigma$-model $(\mathfrak{M},w)$, 
$$\mathfrak{M},w\models_L\phi\ \text{ if and only if }\ \mathfrak{M},w\models_{L'}\psi.$$
Two abstract logic $L$ and $L'$ are equivalent, $L\equiv L'$, if $L\subseteq L'$ and $L'\subseteq L$.

The following definitions single out some conditions which can be satisfied by some abstract logic, in particular by FML.

\begin{defi}\label{abstract2}
\begin{itemize}
\item (Closure under boolean and modal connectives) For any language $\sigma$, $L$ contains all atomic sentences of $\sigma$ and is closed under boolean connectives $\neg$, $\vee$ and $\wedge$ and modal operators $\square$ and $\lozenge$. Furthermore, the satisfaction of atomic sentences as well as boolean connectives and modal operator have the same satisfaction relation in  $\models_L$ as in FML. 
\item (Closure under quantifiers) For any $\phi\in L[\sigma]$ and each constant $c\in \sigma$ there are sentences $\exists x_c\phi(x_c)$ and $\forall x_c\phi(x_c)$ in $L[\tau\diagdown\{c\}]$  such that for any $\sigma\diagdown\{c\}$-structure $(\mathfrak{M},w)$,
\begin{itemize}
\item $\mathfrak{M},w\models_L \exists x_c\phi(x_c) \  \text{ if and only if } \   \mathfrak{M}_d,w\models\phi, \ \text{ for some } d\in D(w).$ 
\item $\mathfrak{M},w\models_L \forall x_c\phi(x_c) \  \text{ if and only if } \   \mathfrak{M}_d,w\models\phi, \ \text{ for all } d\in D(w),.$
\end{itemize}
(The model $\mathfrak{M}_d$ is an expansion of $\mathfrak{M}$ to a $\sigma$-structure which assigns $d$ as an interpretation of $c$).
\end{itemize}
\end{defi}

Any bijective map $\rho: \sigma\rightarrow \sigma'$ is a \textit{renaming} if it maps any relation symbol to a relation symbol with the same arity and any constant symbol to a constant symbol. For any $\sigma$-structure $\mathfrak{M}$ and any renaming $\rho:\sigma\rightarrow\sigma'$, we can define a corresponding structure $\rho[\mathfrak{M}]\in Str[\sigma']$ which is defined using $\rho$ in the obvious way, as a counterpart to $\mathfrak{M}$. The interpretation of an element $Q\in\sigma'$ is the interpretation of $\rho^{-1}(Q)$.

\begin{defi}\label{renrel}
(Renaming) An abstract logic L has the renaming property if for any  renaming $\rho$ of language $\sigma$ and for any $L[\sigma]$-sentence $\phi$ there is an $L[\rho(\sigma)]$-sentence $\rho(\phi)$ such that for every $\sigma$-structure $(\mathfrak{M},w)$,
$$\mathfrak{M},w\models_L\phi \ \ \ \text{ if and only if } \ \ \ \rho(\mathfrak{M}),w\models_L\rho(\phi).$$
\end{defi}

We say that a constant domain  model $\mathfrak{M}$ is closed under the predicate $P(x,c_1\dots,c_n)$ if for any $w$ and $w'$ in $W$ if $\mathfrak{M},w\models\exists x P(x,c_1,\dots,c_n)$ and $\mathfrak{M},w'\models\exists x P(x,c_1,\dots,c_n)$, then for any $d\in D$ we have $\mathfrak{M},w\models P(d,c_1,\dots,c_n)$ if and only if $\mathfrak{M},w'\models P(d,c_1,\dots,c_n)$.

Next we want to define the notion of relativization property for an abstract logic.
This property is only defined for constant domain Kripke models, since relativization is used in Theorem \ref{lind} which is only valid for constant domain models.
\begin{defi}\label{relativ}
(Relativization) For any $L[\sigma]$-sentence $\phi$ and any predicate $P(x,c_1,\dots,c_n)$ where $P$ and $c_1,\dots,c_n$ are not in $\sigma$, there exists a sentence $\phi^{P}$ in $L[\sigma\cup \{P,c_1,\dots, c_n\}]$, such that for any constant domain $\sigma\cup \{P,c_1,\dots, c_n\}$-structure $(\mathfrak{M},w)$ closed under $P(x,c_1,\dots,c_n)$,
$$\mathfrak{M},w\models_L\phi^{P}\ \text{ if and only if }\ \mathfrak{M}^{P},w\models_L\phi,$$
where $\mathfrak{M}^{P}$ is a $\sigma$ model satisfies, 
\begin{itemize}
\item  $W^{P}=\{w'\in W\ |\ \mathfrak{M},w'\models\exists x P(x,\bar{c})\}$.
\item  $D^{P}=\{ a\in D \ |\ \mathfrak{M},w\models P(a,\bar{c}), \text{ for any } w\in W^P \}$.
\item $R^{P}= R\upharpoonright W^{P}$.
\item For any $w'\in W^{P}$,  $I^{P}(w',-)= I(w',-)\upharpoonright D^{P}$. 
\end{itemize}
\end{defi}

\begin{rem}\emph{
First-order modal logic has the relativization property. To see this, for a given first order modal $\sigma$-sentence $\phi$ and atomic formula $P(x,\bar{c})$ we let the relativization of $\phi$ to $P(x,\bar{c})$ be the sentence $\phi^{P}:=\psi^P\wedge \exists x P(x,\bar{c})$ where $\psi^P$ is defined inductively as follows:
\begin{itemize}
\item If $\phi$ is an atomic sentence, then $\psi^P=\phi$.
\item If $\phi= \exists x\theta(x)$, then $\psi^P= \exists x (P(x,\bar{c})\wedge \theta^P(x))$.
\item If $\phi = \lozenge\theta$, then $\psi^P = \lozenge(\exists x P(x,\bar{c})\wedge \theta^P)$.
\end{itemize}
}\end{rem}

The following  properties highlight some of model theoretic features which are true in FML and can be used for its characterization.

\begin{defi}
Let $(L,\models_L)$ be an abstract logic. 
\begin{itemize}
\item  $L$ has the \textit{(countable) compactness property} if for every (countable) set $T$ of $L[\sigma]$-sentences, $T$ has a model provided that every finite subset of $T$ has a model.
\item  $L$ is \textit{invariant under bisimulations}, if for  every two $\sigma$-structures $(\mathfrak{M},w)$ and $(\mathfrak{N},v)$, such that $(\mathfrak{M},w)\rightleftarrows(\mathfrak{N},v)$ the two structures satisfy the same set of $L$-sentences, that is $\mathfrak{M},w\models_L\phi$ if and only if $\mathfrak{N},v\models_L\phi$, for any $\phi\in L[\sigma]$. 
\item $L$ has the \textit{Tarski union property} (TUP) if for any elementary chain of $\sigma$-structures $\langle \mathfrak{M}_i: i\in\mathbb{N}\rangle$, the union $\bigcup_{i\in\mathbb{N}}\mathfrak{M}_i$ is an elementary extension of each $\mathfrak{M}_i$. 
\end{itemize}
\end{defi}

From now on, since there is no danger of ambiguity, for brevity we omit the subscript $L$ from $\models_L$.

\subsection{Lindstr\"om Theorem}\label{seclind}

In this subsection,  we provide a version of Lindstr\"om's theorem for FML.  
 Here it will be shown that FML with respect to constant domain Kripke models, is a maximal logic which satisfies compactness and bisimulation invariance properties. The proof naturally mimics in a more elaborate way the existing proof for first-order logic. The main proof strategy behind the classical Lindstr\"om theorem is to encode partial isomorphisms between two models in an expanded language. More precisely, by way of contradiction, if a given $\sigma$-sentence $\phi$ in an abstract logic $L$ does not have a first-order equivalent $\sigma$-sentence, then one could obtain two sequence of first-order structures $\langle \mathfrak{M}_i : i\in \mathbb{N}\rangle$ and    $\langle \mathfrak{N}_i : i\in \mathbb{N}\rangle$ such that $\mathfrak{M}_k$ and $\mathfrak{N}_k$ are $k$-partial isomorphic but do not agree on $\phi$. Now one can suitably expand the language $\sigma$ to $\sigma'$ in which the notion of $k$-partial isomorphism for $\sigma$-structures as well as basic properties of the structure $(\mathbb{N} , \leq)$ can be formalized. So after providing such suitable expansion $\sigma'$ using compactness and the L\"owenheim-Skolem property one could obtain two countable $\sigma$-structures $\mathfrak{M}_H$ and $\mathfrak{N}_H$ indexed by a non-standard number $H$ in an elementary expansion of the structure $(\mathbb{N},\leq)$ which are $H$-partial isomorphic. On the other hand, since $H$ is a non-standard element, it follows that $\mathfrak{M}_H$ and $\mathfrak{N}_H$ are partially isomorphic and therefore isomorphic by countability of $\mathfrak{M}_H$ and $\mathfrak{N}_H$ but do not agree on $\phi$. This contradicts the Isomorphism property of abstract logic $L$.
 
Now there is a subtlety for extending this proof to FML. This is due to the fact that in the case of first-order modal logic, the logic cannot quantify over possible worlds in the same way as elements of the domain. To overcome to this difficulty we expand the language with some new predicates for ``worlds'' and ``accessibility relation''. The expanded language provides some new predicates for possible worlds and accessibility relation in addition to components which are provided in the original proof. So this expansion enables us to encode a bisimulation of two models and prove the following theorem. 

Note that in the following theorem we restrict ourselves to constant domain Kripke models and furthermore assume that an abstract logic has the closure properties of Definition \ref{abstract2} and the renaming and the relativization properties in the Definition \ref{renrel} and \ref{relativ}, respectively.

\begin{theo}\label{lind}
Any  abstract logic $L$ containing first-order modal logic is equivalent to FML if and only if it is countably compact and invariant under bisimulations. 
\end{theo}

\begin{proof}
Suppose on the contrary that there is a sentence $\phi\in L$ which is not equivalent to any formula of FML. By the Expansion and Occurrence properties, we can assume that $\sigma=\sigma_\phi$ and therefore $\sigma$ is finite. So by Proposition \ref{ckf}, for each $(l,r)\in\mathbb{N}\times \mathbb{N}$, there are two models $(\mathfrak{M},w)$ and $(\mathfrak{N},v)$ which are $\langle l,r\rangle$-bisimilar but do not agree on $\phi$. By bisimulation invariance property we can assume that $(\mathfrak{M},w)$ and $(\mathfrak{N},v)$ are point-generated models.

So  for each $k\in\mathbb{N}$, take $(\mathfrak{M}_k,w_k)$ and $(\mathfrak{N}_k,v_k)$ to be two pointed $k$-bisimilar models such that $\mathfrak{M}_k,w_k\models_L\phi$ but $\mathfrak{N}_k,v_k\models_L\neg\phi$.

Like the proof of Lindstr\"om's theorem for first-order logic (see Theorem 2.5.4 in \cite{chang:model73}), we want to construct a new model $(\mathfrak{C},c_0)$ which has $\mathfrak{M}$ and $\mathfrak{N}$ as a submodel and we would like to be able to say that for each $k\in \mathbb{N}$, a model $\mathfrak{C}$ has two $k$-bisimilar submodels which  disagree on $\phi$. To this end, we introduce a suitable expanded language $\sigma_{Lind}=\sigma_{skeleton}\cup \sigma_{bisim}$  which is introduced to prove the desired properties. 
By the Isomorphism property, we may further suppose that all  the constants of $\sigma$ have the same interpretation in all the $\mathfrak{M}_k$. Moreover, their domains $D_k$, are disjoint from $\mathbb{N}$. Now let $\mathfrak{M}$ be the union of $\mathfrak{M}_k$s which has each $\mathfrak{M}_k$ as a submodel. Respectively, $\mathfrak{N}$ can be constructed in a similar way by the additional assumption that the interpretation of each constant of $\sigma$ is equal to the interpretation of it in $\mathfrak{M}$. 

\textbf{First step, defining \boldmath{$\sigma_{skeleton}$}:} A $\sigma_{skeleton}$ carries out the task of extracting the frame structure of a $\sigma$-model and encoding it as a part of a domain of a  $\sigma_{skeleton}$-model. 

Hence we need to add the set of worlds and the set of natural number to the domain. Furthermore, we want to be able to distinguish between the elements of each of these sets, so one may add three new different predicates to the language for each of the sets, domain, worlds, and index. To avoid further complications, and make our sentences simpler, we present a language $\sigma_{Lind}$ as a many-sorted language instead of using three above predicates.

A $\sigma_{skeleton}$ has three sorts  ``world'', ``object'' and ``index''. We use the letters $\underline{w},\underline{v},\underline{w}',\underline{v}'$ for variables of the world sort, $x,y,x',y'$ for variables of the object sort  and letters $k,l,n,r,s$ for variables of the index sort. 

Let $\sigma_\mathfrak{M}$ be the language consisting of five new predicates, $(world_\mathfrak{M},2)$, $(root_\mathfrak{M},2)$, $(dom_\mathfrak{M},2)$ and $(rel_\mathfrak{M},3)$ and $(mod_\mathfrak{M},2)$.

Expand the model $\mathfrak{M}=(W_\mathfrak{M},R_\mathfrak{M},D_\mathfrak{M}, I_\mathfrak{M})$ to the model $\mathfrak{M}'=(W_\mathfrak{M}, R_\mathfrak{M}, D'_\mathfrak{M}, I'_\mathfrak{M})$ in $\sigma\cup\sigma_\mathfrak{M}$ as follows. Put $D'_\mathfrak{M}= D_{\mathfrak{M}} \cup W_\mathfrak{M} \cup\mathbb{N}$. (Note that $W_\mathfrak{M}=\bigcup_{n\in \mathbb{N}} W_{\mathfrak{M}_n}$.) While $D_{\mathfrak{M}}$ is the interpretation of object sorts, $W_\mathfrak{M}$ is for world sorts and $\mathbb{N}$ is for index sorts.  
An element $w$ of  $W_\mathfrak{M}$ in $D'_\mathfrak{M}$ is denoted by $\textbf{w}$.
To describe $I'_\mathfrak{M}$ we need to give interpretations to new predicates.
\begin{itemize}
\item $world_\mathfrak{M}(\textbf{w},m)$ is true in $w$ if and only if $w \in W_{\mathfrak{M}_m}$ that is $I'_\mathfrak{M}(w,world_\mathfrak{M})=\{(\textbf{w},m)\}$.
\item $root_\mathfrak{M}(\textbf{w},m)$ holds exactly in $w$ when $w$ is the root of $\mathfrak{M}_m$.
\item In any $w\in W_{\mathfrak{M}_m}$, $dom_\mathfrak{M}(a,m)$ is true in $w$ if  $a\in D_{\mathfrak{M}_m}$.
\item $rel_\mathfrak{M}(\textbf{w},\textbf{w}',m)$ is true in $w$ if and only if $wR_{\mathfrak{M}_m}w'$. 
\item $I'_\mathfrak{M}(w,mod_\mathfrak{M})=\{ (a,m) \ | \ a\in D_{\mathfrak{M}_m}(w)\}$ for all $w\in W_{\mathfrak{M}_m}$.
\end{itemize}

In a similar way the language $\sigma_\mathfrak{N}$ and subsequently $\mathfrak{N}'$ can be constructed. 

Let $\sigma_{skeleton}$ be the language which adds to $\sigma \cup \sigma_\mathfrak{M} \cup \sigma_\mathfrak{N}$ one constant symbol $\lambda$ and three  new binary predicates $dom_\mathfrak{M}^\star$ and $dom_\mathfrak{N}^\star$ and $\leq$ together with a new $(n+1)$-ary predicate symbol $P^\star$ for each $n$-ary predicate symbol $P$ in $\sigma$.

Now construct a new pointed Kripke $\sigma_{skeleton}$-model $(\mathfrak{C},c_0)=\left(W_\mathfrak{C} , R_\mathfrak{C} , D_\mathfrak{C} ,  I_\mathfrak{C},  c_0\right)$ as follows:
$$W_\mathfrak{C}:=W_\mathfrak{M}\uplus W_\mathfrak{N}\uplus\{c_0\},$$ 
$$D_\mathfrak{C}:= D'_\mathfrak{M} \cup D'_\mathfrak{N} \cup \bigcup_{n\in \mathbb{N}} D^\star_{\mathfrak{M}_n} \cup \bigcup_{n\in \mathbb{N}}D^\star_{\mathfrak{N}_n},\footnote{Recall that $D^\star $ is the set of all finite subset of a set $D$} $$  
and 
$$R_\mathfrak{C}:= R_\mathfrak{M} \uplus R_\mathfrak{N} \uplus \{(c_0,w)\ |\ w\in W_\mathfrak{M}\uplus W_\mathfrak{N}\}.$$

For any $w\in W_\mathfrak{M}$ (resp. $w\in W_\mathfrak{N}$), all the symbols of $\sigma\cup \sigma_\mathfrak{M}$ (resp. $\sigma \cup \sigma_\mathfrak{N}$) have the same interpretation as in $\mathfrak{M}'$ (resp. $\mathfrak{N}'$). Also for any $n$-ary predicate $P$ in $\sigma$ and $w\in W_\mathfrak{M}$,  $I(w,P^\star)=\{(\textbf{w},a_1,\dots,a_n)\ |\ (a_1,\dots,a_n)\in I_\mathfrak{M}(P,w)\}$. Similarly interpret any $P^\star$ in any $w\in W_N$. The rest of symbols are interpreted by empty set, whenever $w\in W_M\uplus W_N$. 

Now we aim that the interpretation of symbols $\sigma_\mathfrak{M}\cup \sigma_\mathfrak{N}$ except the predicates $mod_\mathfrak{M}$ and $mod_\mathfrak{N}$ in $c_0$ reflects the structure $\mathfrak{M}\uplus \mathfrak{N}$. So  for instance we define
\begin{itemize}
\item  $(\textbf{w},m)\in I_\mathfrak{C}(c_0,world_\mathfrak{M})$ if and only if $w\in W_{\mathfrak{M}_m}$,
\item and $rel_\mathfrak{M}(\textbf{w},\textbf{w}',m)$ is true in $c_0$ if and only if $w R_{\mathfrak{M}_m}w'$.
\end{itemize} 

The other symbols can be interpreted similarly. 
 
 Put  $I_\mathfrak{C}(c_0, P^\star)=\{ (\textbf{w}, a_1,\dots ,a_n)\ |\ (a_1,\dots ,a_n)\in I_\mathfrak{C}(w,P)\}$, $I(c_0,\leq)=\{(m,m')\ |\  m\leq m'\}$, and $I_\mathfrak{C}(c_0,dom_{\mathfrak{M}}^\star)=\{(\bar{a},m)\ |\ \bar{a}\in D_{\mathfrak{M}_m}^\star\}$. Similar interpretation is used for the predicate $dom_\mathfrak{N}^\star$ in $c_0$.
 
Finally, let the interpretation of constant $\lambda$ be the empty sequence in $D^\star= \bigcup_{n\in \mathbb{N}} D^\star_{\mathfrak{M}_n} \cup \bigcup_{n\in \mathbb{N}}D^\star_{\mathfrak{N}_n}$.
  
\textbf{Second step,  defining \boldmath{$\sigma_{bisim}$}:} $\sigma_{bisim}$ converts the notion of $k$-bisimilarity of two $\sigma$-models into some first-order  $\sigma_{bisim}$-statements.   
 
Let $\sigma_{bisim}= \{ =,(Z,6),\ (^\frown,3),\ (\langle,\rangle,3),\ (\textit{left},2),\ (\textit{right},2)\}$.  Except for $=, Z$ the other relation symbols are intended to be interpreted as graphs of functions which are needed to encode the notion of bisimilarity. The relation $^\frown$ is reserved for concatenation function on $D_{\mathfrak{M}_n}^\star$ and $D_{\mathfrak{N}_n}^\star$ for all $n\in \mathbb{N}$  which is interpreted as 
$$ ^\frown(\bar{a},\bar{b},\bar{c}) \ \text{ if and only if } \bar{a},\bar{b},\bar{c}\in D_{\mathfrak{M}_n}^\star \ \text{ and }  \bar{c}=\bar{a}\bar{b}.$$
The other relations are implemented for the coding function. To ease the notations we interpret them as function symbols.

Let $\sigma_{Lind}=\sigma_{skeleton}\cup \sigma_{bisim}$.
Now we expand the $\sigma_{skeleton}$-model $\mathfrak{C}$ to $\sigma_{Lind}$-model $\mathfrak{C}'=(W_\mathfrak{C}, R_\mathfrak{C}, D_\mathfrak{C}, I_\mathfrak{C'},c_0)$ by interpreting the corresponding relations.
 
The interpretation of $I_{\mathfrak{C}'}(Z,c_0)$ is the set of all 6-tuples $(\textbf{w},\bar{a},\textbf{v},\bar{b},l,k)$, where $w\in W_{\mathfrak{M}_l}$, $\bar{a}\in D^\star_{\mathfrak{M}_l}$ and $v\in W_{\mathfrak{N}_l}$, $\bar{b}\in D^\star_{\mathfrak{N}_l}$ and there exit $a_1,\dots,a_n\in D_{\mathfrak{M}_l}$ and $b_1,\dots,b_n\in D_{\mathfrak{N}_l}$ such that $a_1^\frown\dots ^\frown a_n=\bar{a}$, $b_1^\frown\dots ^\frown b_n=\bar{b}$,  and $wa_1\dots a_n\rightleftarrows_k vb_1\dots b_n$.

The other symbols can be interpreted naturally in $c_0$ as defined in the preceding section (see section \ref{sec bisim}).

For the rest of possible worlds $w\in W_\mathfrak{C'}$ interpret elements of $\sigma_{bisim}$ as empty set.

\textbf{Third step, introducing \boldmath{$L[\sigma_{Lind}]$}-theory \boldmath{$T_{Lind}$}:} In this step we are going to introduce an $L[\sigma_{Lind}]$-theory to formalize the notion of bisimilarity. To proceed with this idea we introduce a sequence of $L[\sigma_{Lind}]$-sentences. Let $\mathfrak{f,g}\in\{\mathfrak{M},\mathfrak{N}\}$ and $\mathfrak{f}\neq\mathfrak{g}$.

\begin{enumerate}
\item The following sentences state that for any world from $W_\mathfrak{M}$ and $W_\mathfrak{N}$ in the frame there is a unique element from world sort that satisfies either $world_\mathfrak{M}$ or $world_\mathfrak{N}$ for some $n$ respectively, and for each element of world sort there is a corresponding world in the frame.
  \begin{enumerate}
   \item
      $\forall \underline{w}\forall n\ (\ (world_\mathfrak{M}(\underline{w},n)\leftrightarrow \lozenge world_\mathfrak{M}(\underline{w},n))\ \wedge \\
      ~ \hspace{1.5cm} (world_\mathfrak{N}(\underline{w},n)\leftrightarrow \lozenge world_\mathfrak{N}(\underline{w},n))\ ).$
      
   \item 
        $ \forall\underline{w}\forall n\ \neg(world_\mathfrak{M}(\underline{w},n) \wedge world_\mathfrak{N}(\underline{w},n)).$
   \item
    $\square\  (\ \exists\underline{w} \exists n\  ((world_\mathfrak{M}(\underline{w},n) \vee world_\mathfrak{N}(\underline{w},n)) \wedge \neg (world_\mathfrak{M}(\underline{w},n) \wedge world_\mathfrak{N}(\underline{w},n)) \\
 ~ \hspace{1.1cm} \wedge  \forall\underline{w}'\forall n' ((\underline{w}\neq \underline{w}' \vee n \neq n') \rightarrow (\neg world_\mathfrak{M}(\underline{w}',n') \wedge \neg world_\mathfrak{N}(\underline{w}',n')))).$
    \end{enumerate}

\item Any accessible world $w'$ from $w\in W_{\mathfrak{f}_n}$ satisfies $world_\mathfrak{f}(\textbf{w}',n)$  and is accessible from $c_0$.
  \begin{enumerate}
   \item
       $\forall\underline{w} \forall n \square (world_\mathfrak{f}(\underline{w},n)\rightarrow \square \exists \underline{w}' world_\mathfrak{f}(\underline{w}',n)).$
   \item
        $\forall\underline{w} \forall n \forall\underline{w}' (\lozenge(world_\mathfrak{f}(\underline{w},n)\wedge \lozenge world_\mathfrak{f}(\underline{w}',n)) \rightarrow \lozenge world_\mathfrak{f}(\underline{w}',n)).$

   \end{enumerate}
\item There is a one-to-one corresponding between the domain of $D_{\mathfrak{f}_n}$ 
and the elements from domain sort which satisfy $dom_\mathfrak{f}(-,n)$. 
Moreover $(\mathfrak{C}',c_0)$ is closed under $dom_\mathfrak{f}(-,n)$. 
     \begin{enumerate}
       \item
       $\forall n\forall x\ (dom_\mathfrak{f}(x,n) \rightarrow \square\ (\exists \underline{w}\ world_\mathfrak{f}(\underline{w},n)\rightarrow dom_\mathfrak{f}(x,n))).$
      \item
       $\forall\underline{w}\forall n\forall x\ (\lozenge (world_\mathfrak{f}(\underline{w},n)\wedge dom_\mathfrak{f}(x,n)) \rightarrow dom_\mathfrak{f}(x,n)).$
      \item
         $\forall \underline{w}\forall \underline{w}'\forall n \forall x\  (\ \square\ (world_\mathfrak{f}(\underline{w},n) \rightarrow dom_\mathfrak{f}(x,n)) \leftrightarrow\\
 ~ \hspace{2.7cm}  \square(world_\mathfrak{f}(\underline{w}',n) \rightarrow dom_\mathfrak{f}(x,n))).$       
       \item
       $\forall\underline{w}\forall n\ \square (world_\mathfrak{f}(\underline{w},n)\rightarrow \forall n' (\neg \exists x\ dom_\mathfrak{g}(x,n') \wedge (n\neq n' \rightarrow \neg\exists x\ dom_\mathfrak{f}(x,n'))).$

        \item The domain of models $\mathfrak{M}_n$ and $\mathfrak{N}_n$, for any $n\in \mathbb{N}$, is non-empty,
       
       $\forall n\ (\exists x\  dom_\mathfrak{M}(x,n) \wedge \exists x\  dom_\mathfrak{N}(x,n)).$

       \item The interpretation of any constant $c\in \sigma$ is in the domain of all $\mathfrak{M}_n$ and $\mathfrak{N}_n$, for all $n$.

   $\forall x \ (x=c\rightarrow \forall n (dom_\mathfrak{M}(x,n)\wedge dom_\mathfrak{N}(x,n))).$
     \end{enumerate}
\item For any $n\in \mathbb{N}$ the predicate $mod_\mathfrak{f}(-,n)$ 
specifies model $\mathfrak{f}_n$.

$\forall n\ \square \ ( \exists\underline{w}\  world_\mathfrak{f}(\underline{w},n)\ \rightarrow \ \forall x (dom_\mathfrak{f}(x,n)\leftrightarrow mod_\mathfrak{f}(x,n))\ \wedge\\ 
 ~ \hspace{5.2cm} \forall n' (\neg\ \exists x\ mod_\mathfrak{g}(x,n') \wedge (n\neq n' \rightarrow \neg \exists x\ mod_\mathfrak{f}(x,n')))).$

\item The atomic formula $P(a_1,\dots,a_n)$ is true in the world $w$ if and only if $P^\star(\textbf{w},a_1,\dots,a_n)$ is true in both $c_0$ and $w$.
     \begin{enumerate}
         \item
              $\forall\underline{w}\forall n\forall x_1\dots \forall x_n\ \square (world_\mathfrak{f}(\underline{w},n)\rightarrow (P(x_1,\dots,x_n) \leftrightarrow P^\star(\underline{w},x_1,\dots,x_n))).$
         \item
             $\forall\underline{w}\forall x_1\dots \forall x_n\ (P^\star(\underline{w},x_1,\dots,x_n) \rightarrow \square (\exists n\ world_\mathfrak{f}(\underline{w},n)\rightarrow P(x_1,\dots,x_n)).$
         \item
             $\forall\underline{w}\forall n\forall x_1\dots \forall x_n\ (\lozenge (world_\mathfrak{f}(\underline{w},n)\wedge P(x_1,\dots,x_n))\rightarrow P^\star(\underline{w},x_1,\dots,x_n)).$
       \end{enumerate}

\item There is a back and forth property between the relation $R_{\mathfrak{f}_n}$ and the predicate $rel_\mathfrak{f}(-,-,n)$.
    \begin{enumerate}
       \item
           $\forall\underline{w}\forall n\ \square (world_\mathfrak{f}(\underline{w},n) \rightarrow \forall \underline{w}' (rel_\mathfrak{f}(\underline{w},\underline{w}',n) \leftrightarrow \lozenge world_\mathfrak{f}(\underline{w}',n))).$
        \item
            $\forall\underline{w} \forall \underline{w}'\forall n\ (rel_\mathfrak{f}(\underline{w},\underline{w}',n) \rightarrow \square (world_\mathfrak{f}(\underline{w},n) \rightarrow rel_\mathfrak{f}(\underline{w},\underline{w}',n))).$
        \item
           $\forall\underline{w} \forall \underline{w}'\forall n\ (\lozenge (world_\mathfrak{f}(\underline{w},n) \wedge rel_\mathfrak{f}(\underline{w},\underline{w}',n))\rightarrow rel_\mathfrak{f}(\underline{w},\underline{w}',n)).$
      \end{enumerate}

\item The predicate $dom^\star_\mathfrak{f}(-,n)$ 
 indicates the elements of $D^\star_{\mathfrak{f}_n}$ and $\lambda$ is a constant which is interpreted as the empty sequence in $D^\star$.
       \begin{itemize}
          \item $\forall n \forall x_1,\dots ,\forall x_k\ (\bigwedge_{1\leq i\leq k} dom_\mathfrak{f}(x_i,n) \leftrightarrow dom^\star_\mathfrak{f}(x_1^\frown\dots^\frown x_k,n)).$
          \item $\forall x\forall y (x^\frown y=\lambda \rightarrow x=\lambda \wedge y=\lambda).$
       \end{itemize}

\item For any $n\in \mathbb{N}$, each of $\mathfrak{f}_n$ has a root.
     \begin{enumerate}
        \item 
$\forall n \exists\underline{w}\ (\lozenge root_\mathfrak{f}(\underline{w},n) \wedge \forall\underline{w}' (\underline{w}\neq\underline{w}' \rightarrow \neg \lozenge root_\mathfrak{f}(\underline{w}',n))).$
          \item
$\forall n \forall\underline{w}\ \square (root_\mathfrak{f}(\underline{w},n)\rightarrow world_\mathfrak{f}(\underline{w},n)).$
      \end{enumerate}
    
\item For each $n\in \mathbb{N}$, while the $\mathfrak{M}_n$'s root satisfies $\phi$, the $\mathfrak{N}_n$'s root refutes $\phi$.\label{disagree}
    
    $\forall n\ \forall \underline{w}\ (\square (root_\mathfrak{M}(\underline{w},n)\rightarrow\phi^{mod_\mathfrak{M}(x,n)})\ \wedge \ 
  \square  (root_\mathfrak{N}(\underline{w},n)\rightarrow\neg\phi^{mod_\mathfrak{N}(x,n)})). \footnote{Note that the model $(\mathfrak{C}',c_0)$ is closed under both $mod_\mathfrak{M}(x,n)$ and $mod_\mathfrak{N}(x,n)$. }$

\item The following $\sigma_{Lind}$-sentences state the basic properties of bisimilarity in the language $\sigma_{Lind}$ and all of them are true in $(\mathfrak{C}',c_0)$.\label{bisimcod}
\begin{enumerate}
\item 
$\forall \underline{w}, \underline{v}, \forall x_1,\dots ,x_n, y_1\dots ,y_n, \forall l,k \ ~ (~ Z(\underline{w},x_1^\frown\dots ^\frown x_n,\underline{v},y_1^\frown\dots ^\frown y_n,l,k)   \rightarrow   \\
~  \hspace{6.3cm}       ( \textit{left}(k)\leq \textit{left}(l) \wedge \\
 ~ \hspace{6.3cm}                    \textit{right}(k)\leq \textit{right}(l) \ \wedge \\
~   \hspace{6.3cm} world_\mathfrak{M}(\underline{w},l) \wedge world_\mathfrak{N}(\underline{v} ,l) \wedge \\
 ~    \hspace{6.2cm}  \bigwedge_{1\leq i\leq n} dom_\mathfrak{M}(x_i,l) \wedge \bigwedge_{1\leq i\leq n} dom_\mathfrak{N}(y_i,l)).$

\item Two bisimilar worlds $w$ and $v$ satisfy the same atomic formulas.

$\forall \underline{w},\underline{v},\forall x_1,\dots , x_n ,y_1,\dots ,y_n,\forall k,l \  (\ Z(\underline{w}, x_1^\frown\dots ^\frown x_n,\underline{v}, y_1^\frown\dots ^\frown y_n ,l,k)\ \rightarrow\\ 
 ~ \hspace{6.2cm}  (P^\star(\underline{w},x_1,\dots , x_n)\leftrightarrow P^\star(\underline{v},y_1,\dots ,y_n))). $

For all predicates $P\in\sigma$.
\item  Frame back and forth is the conjunction of the following sentences.
\begin{description}
\item[$\lozenge$-Forth.]
$ \forall \underline{w},\underline{v},\underline{w}' ,\forall x_1,\dots ,x_n, y_1\dots ,y_n,\forall l,k\\
 ~  (\ Z(\underline{w}, x_1^\frown\dots ^\frown x_n,\underline{v}, y_1^\frown\dots ^\frown y_n,l,k)\ \wedge\   rel_\mathfrak{M}(\underline{w},\underline{w}' ,l) \ \wedge\\
 ~ \hspace{.4cm} \exists s~ \textit{left}(k)=s+1 \ \ \ \rightarrow \\
  ~    \exists \underline{v}'\ (rel_\mathfrak{N}(\underline{v} , \underline{v}',l)\ \wedge\ Z(\underline{w}',x_1^\frown\dots ^\frown x_n, \underline{v}',y_1^\frown\dots ^\frown y_n,l,\langle s,\textit{right}(k)\rangle))).$

\item[$\lozenge$-Back.]
$\forall \underline{w},\underline{v},\underline{v}' ,\forall x_1,\dots ,x_n, y_1\dots ,y_n,\forall l,k\\
 ~  (\ Z(\underline{w}, x_1^\frown\dots ^\frown x_n,\underline{v}, y_1^\frown\dots ^\frown y_n,l,k)\ \wedge\  rel_\mathfrak{N}(\underline{v},\underline{v}' ,l)\ \wedge \\
 ~ \hspace{.4cm} \exists s~ \textit{left}(k)=s+1 \ \ \ \rightarrow \\
  ~  \exists \underline{w}'\ (rel_\mathfrak{M}(\underline{w} , \underline{w}',l)\ \wedge\ Z(\underline{w}',x_1^\frown\dots ^\frown x_n, \underline{v}',y_1^\frown\dots ^\frown y_n,l,\langle s,\textit{right}(k)\rangle))).$

\end{description}
\item Domain back and forth is the conjunction  of the following sentences. 
\begin{description}
\item[$\exists$-Forth.]
$\forall \underline{w} ,\underline{v},\forall x_1,\dots ,x_n,\forall y_1,\dots ,y_n,\forall x, \forall l,k \\
 ~   (\ Z(\underline{w} ,x_1^\frown\dots ^\frown x_n,\underline{v} ,y_1^\frown\dots ^\frown y_n,l,k)\ \wedge\   dom_\mathfrak{M}(x,l)\ \wedge \\
 ~ \hspace{.4cm} \exists s ~ \textit{right}(k)=s+1\ \ \  \rightarrow\\ 
  ~   \exists y\ (dom_\mathfrak{N}(y,l)\ \wedge\   Z(\underline{w},x_1^\frown\dots ^\frown x_n^\frown x,\underline{v},y_1^\frown\dots ^\frown y_n^\frown y,l,\langle \textit{left}(k),s\rangle))).$

\item[$\exists$-Back.] 
$\forall \underline{w} ,\underline{v},\forall x_1,\dots ,x_n,\forall y_1,\dots ,y_n,\forall y, \forall l,k\\
 ~ (\ Z(\underline{w} ,x_1^\frown\dots ^\frown x_n,\underline{v} ,y_1^\frown\dots ^\frown y_n,l,k)\ \wedge\  dom_\mathfrak{N}(y,l)\ \wedge\\
 ~ \hspace{.4cm} \exists s ~ \textit{right}(k)=s+1\ \ \ \  \rightarrow\\
  ~   \exists x\ (dom_\mathfrak{M}(x,l)\ \wedge \  Z(\underline{w},x_1^\frown\dots ^\frown x_n^\frown x,\underline{v},y_1^\frown\dots ^\frown y_n^\frown y,l,\langle \textit{left}(k),s\rangle))).
$
\end{description}
\end{enumerate}
\item Since for each $n\in \mathbb{N}$ $(\mathfrak{M}_n,w_n)\rightleftarrows_n(\mathfrak{N}_n,v_n)$, it follows that \label{root bisim}

$ \mathfrak{C}',c_0\models \forall n \forall \underline{w},\underline{v} \ (\lozenge root_\mathfrak{M}(\underline{w},n)\ \wedge\ \lozenge root_\mathfrak{N}(\underline{v},n)\ \rightarrow\ Z(\underline{w},\lambda,\underline{v},\lambda,n,n)).$

\item  $\sigma_{Lind}$ statements indicate the basic properties of coding that is

 $\langle,\rangle$ is a bijective function and $\textit{left}$ and $\textit{right}$ are unique functions such that $\langle \textit{left}(n),\textit{right}(n)\rangle=n$, for each $n\in\mathbb{N}$.

\item $\frown$ is a concatination function.

\item Finally, formalize the basic properties of order $\leq$ by the following sentences: \label{rootexist}

\begin{itemize}
\item 
$\mathbb{N}$ under $\leq$ is a linear order with well-defined successors and predecessors except for the first element. 
\item $\alpha_0= \forall n\ \exists \underline{w},\underline{v}\  (\lozenge root_\mathfrak{M}(\underline{w},n)\ \wedge\ \lozenge root_\mathfrak{N}(\underline{v},n))$.
\item $\alpha_1= \exists n\   1\leq n$.
\item $\alpha_2= \exists n\  (1\leq n\ \wedge\  2\leq n)$.\\
\vdots
\end{itemize}
\end{enumerate}

Having defined the above statements, let $T_{Lind}$ be the $L[\sigma_{Lind}]$-theory consisting of sentences above.

Note that any finite subset of $T_{Lind}$ is satisfiable in $(\mathfrak{C'},c_0)$. Therefore, by the compactness property, $T_{Lind}$ has a model $(\mathfrak{E},e)$. Choose  two non-standard elements $H_1$ and $H_2$ in $(\mathbb{N}^{\mathfrak{E}},\leq^{\mathfrak{E}})$ and for $H=\langle H_1,H_2\rangle$, define two models $(\hat{\mathfrak{M}}_H,\hat{w}_H)$ and $(\hat{\mathfrak{N}}_H,\hat{v}_H) $ as follows. (Subsequently, from $(\mathfrak{E},e)$ one can extract two $\sigma$-models.)

\begin{itemize}
\item $W_{\hat{\mathfrak{M}}_H}=\{w\in W_{\mathfrak{E}}\ |\ \mathfrak{E},w\models \exists \underline{w}\ world_\mathfrak{M}(\underline{w},H)\}$.
\item  $R_{\hat{\mathfrak{M}}_H}  = R_{\mathfrak{E}}\upharpoonright W_{\hat{\mathfrak{M}}_H}$.
\item $D_{\hat{\mathfrak{M}}_H}=\{a\in D_\mathfrak{E} \ | \ \mathfrak{E},w\models dom_\mathfrak{M}(a,H)\ \text{ for all } w\in W_{\hat{\mathfrak{M}}_H}\}$.
\item For each  $w\in W_{\hat{\mathfrak{M}}_H}$, and any $a_1,\dots,a_n$ in $D_{\hat{\mathfrak{M}}_H}$ and any predicate $P\in \sigma$, $I_{\hat{\mathfrak{M}}_H}(w,P)\ =\ \{\ (a_1,\dots,a_n)\ |\, \mathfrak{E},w\models P(a_1,\dots,a_n)\}$. Also, $I_{\hat{\mathfrak{M}}_H}(c)=I_\mathfrak{E}(c)$, for each constant symbol $c\in\sigma$.
\item $\hat{w}_H$ is a world in $W_{\hat{\mathfrak{M}}_H}$ such that $\mathfrak{E},\hat{w}_H\models \exists\underline{w}\ root_\mathfrak{M}(\underline{w},H)$.
\end{itemize}
Note that the existence of $\hat{w}_H$ is guaranteed by the  statement \ref{rootexist}. We may similarly define $(\hat{\mathfrak{N}}_H,\hat{v}_H)$.

Using the sentence \ref{root bisim} which is included in $T_{Lind}$ one can see that  for $(\hat{\mathfrak{M}}_H,\hat{w}_H)$ and $(\hat{\mathfrak{N}}_H,\hat{v}_H)$ the relation $Z(\hat{\textbf{w}}_H,\lambda,\hat{\textbf{v}}_H,\lambda,H,H)$ holds, $\hat{\textbf{w}}_H$ and $\hat{\textbf{v}}_H$ are those elements of $D_\mathfrak{E}$ where $\mathfrak{E},\hat{w}_H\models world_\mathfrak{M}(\hat{\textbf{w}}_H,H)$ and    $\mathfrak{E},\hat{v}_H\models world_\mathfrak{N}(\hat{\textbf{v}}_H,H)$. Also, the sentence \ref{disagree} implies that  $(\hat{\mathfrak{M}}_H,\hat{w}_H)\models\phi$, while $(\hat{\mathfrak{N}}_H,\hat{v}_H)\models\neg\phi$. 

Define a relation $Z'$ between  $(\hat{\mathfrak{M}}_H,\hat{w}_H)$ and $(\hat{\mathfrak{N}}_H,\hat{v}_H)$ as follows.
\begin{eqnarray*}
Z'= &\{&(wa_1\dots a_n,vb_1\dots b_n)\, \, |\, \, n\geq 0,\ \mathfrak{E},e_0\models Z(\textbf{w},a_1^\frown \dots ^\frown a_n,\textbf{v},b_1^\frown \dots ^\frown b_n,H,K)\\
  & & \text{ for some } k\geq 0\}.
\end{eqnarray*}

Note that by the statement \ref{root bisim} $(\hat{w}_H,\hat{v}_H)\in Z'$ so $Z'\neq\emptyset$.
The statements appearing in \ref{bisimcod} guarantee that $Z'$ is a bisimulation between $(\hat{\mathfrak{M}}_H,\hat{w}_H)$ and $(\hat{\mathfrak{N}}_H,\hat{v}_H)$.
\end{proof}

The following theorem can be seen as a strengthening of van Benthem invariance theorem (Theorem \ref{invar})  (see Corollary 3.4 in \cite{benthem:lind09} for propositional case).

\begin{coro}
Let $L$ be a compact abstract logic extending first-order modal logic. Then FML is the bisimulation invariant fragment of $L$. In particular, FML is the bisimulation invariant fragment of first-order logic. 
\end{coro}

\subsection{Some Other Lindstr\"om Theorem}\label{other lind}

The method of proving theorem \ref{lind} could not be applied for varying domain Kripke models to give a general Lindstr\"om theorem for first-order modal logic. So we have to look for some other versions of the Lindstr\"om theorem which can be used to give a characterization theorem related to varying domain Kripke models.

The other versions of the Lindstr\"om theorem, for example Theorem 2.2.1 in \cite{flum:char85} and Theorem 6 in \cite{otto:lind08} use a variant of the Tarski union property. Here we extend this results to give a characterizing theorem for FML based on the Tarski union property. To be more precise,  we extend the method used by Piro and Otto in \cite{otto:lind08} for proving the Lindstr\"om theorem for modal logic with global modality. We assume that an abstract logic has the closure properties of Definition \ref{abstract2}.

First of all, note that FML has the Tarski union property (TUP). 
Let $\langle I,\leq\rangle$ be a linear order and $\langle \mathfrak{M}_i: i\in I\rangle$ be a chain of varying domain Kripke models;  $\mathfrak{M}_i\subseteq \mathfrak{M}_j$ for all $i,j\in I$ with $i\leq j$. 
The union of chain $\langle \mathfrak{M}_i: i\in I\rangle$
 abbreviated by $\mathfrak{M}_\infty$,
is a model $\bigcup_{i\in I}\mathfrak{M}_i =(\bigcup_{i\in I} W_i, \bigcup_{i\in I} R_i, \bigcup_{i\in I} D_i, \bigcup_{i\in I} I_i, \{\bigcup_{i\in I} D_i(w)\})$. 
Since $M_i\subseteq M_j$, the above definition is well defined. 
 
 We say that $\langle \mathfrak{M}_i: i\in I\rangle$ is an elementary chain when for each $i,j\in I$ with $i\leq j$, $\mathfrak{M}_i\preceq\mathfrak{M}_j$.
 
The following proposition is routine and is shown by an easy induction on formulas.

\begin{prop}\label{tup}
For any elementary chain of Kripke models $\langle \mathfrak{M}_i: i\in I\rangle$, $\mathfrak{M}_\infty$ is the elementary extension of each $\mathfrak{M}_i$.
\end{prop}

Note that by the same argument used in Proposition 2.15 \cite{blac:modal01}, one may prove that first-order modal logic has the tree model property, that is any satisfiable first-order modal theory is also  satisfiable in a tree-like model. In fact, any pointed Kripke model $(\mathfrak{M},w)$ is bisimilar to a model $(\mathfrak{N},v_0)$ whose frame $(W_\mathfrak{N},R_\mathfrak{N})$ is a tree with a countable depth that is for every world $v\in W_\mathfrak{N}$ there is a finite path from the root $v_0$ to $v$. 

Let $Th_L(\mathfrak{M},w)$ denote the theory of $(\mathfrak{M},w)$ in $L$. We omit the subscript $L$ if the logic $L$ is FML.

Suppose $\mathfrak{M}$ and $\mathfrak{N}$ are two $\sigma$-models. $\mathfrak{N}$ is an $L$-elementary extension of $\mathfrak{M}$, denoted by $\mathfrak{M} \preceq_L \mathfrak{N}$, if $\mathfrak{M}$ is a submodel of $\mathfrak{N}$ and for any $w\in W_\mathfrak{M}$ and any finite set $A=\{a_1,\dots,a_n\}\subseteq D_\mathfrak{M}$, $\mathfrak{M}_A,w\models \phi$ if and only if $\mathfrak{N}_A,w\models\phi$ for any $\sigma\cup\{c_1,\dots,c_n\}$-sentence $\phi$ in $L$.

\begin{prop}\label{real}
Let $L$ be a compact abstract logic extending FML. Then for every tree-like model $(\mathfrak{M},w_0)$ whose frame has a countable depth  there is an $L$-elementary extension $(\mathfrak{N},v_0)$ which realizes all types of $(\mathfrak{M},w)$, for all $w$ in $\mathfrak{M}$.
\end{prop}
\begin{proof}
Let 
\begin{eqnarray*}
\sigma' := \sigma  &\cup & \{(P_w,1) \mid w\in W_\mathfrak{M}\}\\
                   & \cup & \{(P_{w,\Gamma},1) \mid w\in W_\mathfrak{M} \text{ and } \Gamma \text{ is a modal-type of } \mathfrak{M}\}\\
                 & \cup & \{(P_{w,\Gamma},n)\mid w\in W_\mathfrak{M} \text{ and } \Gamma \text{ is an } \exists\text{-type of } \mathfrak{M} \text{ with } n \text{ variable } \}\\
                 &\cup & \{c_d\ |\ d\in D_\mathfrak{M}\}.
\end{eqnarray*}

Since $(W_\mathfrak{M},R_\mathfrak{M})$ is a countable depth tree, for every $w\in W_\mathfrak{M}$ there is a unique finite path from $w_0$ (the root) to $w$ with the depth $n_w$.  
Now for any $w,w'\in W$ and any $n\in \mathbb{N}$ consider the following sentences:
\begin{enumerate}
\item $\lozenge^{n_w}\forall x P_{w}(x)$. \label{ele1}
\item $\square^{n} (\forall x P_w(x)\longrightarrow\neg\forall x P_{w^\prime}(x))$, if $w'\neq w$.
\item $\square^{n} (\forall x P_w(x)\longrightarrow\lozenge\forall x P_{w^\prime}(x))$, if $(w,w^\prime)\in R$.
\item $\square^{n} (\forall x P_w(x)\longrightarrow\neg\lozenge \forall x P_{w^\prime}(x))$, if $(w,w^\prime)\not\in R$.
\item $\square^{n}(\forall x P_w(x)\longrightarrow\xi)$, for all $\sigma\cup\{c_d\ | d\in D_\mathfrak{M}\}$-sentence $\xi\in Th_{L}(\mathfrak{M},w)$. \label{ele5}
\item $\square^{n} (\forall x P_w(x)\longrightarrow\lozenge \forall x P_{w,\Gamma}(x))$, for all modal-type $\Gamma$ of $(\mathfrak{M},w)$.\label{typ6}
\item $\square^{n} (\forall x P_w(x)\longrightarrow\exists \bar{y} P_{w,\Gamma}(\bar{y}))$, for all $\exists$-type $\Gamma(\bar{y})$ of $(\mathfrak{M},w)$.
\item $\square^{n} \forall \bar{y} (P_{w,\Gamma}(\bar{y})\longrightarrow\phi(\bar{y}))$, for all $\phi(\bar{y})\in\Gamma$ ($\exists$-type) and $P_{w,\Gamma}\in\sigma'$.
\item $\square^{n} \square(\forall x P_{w,\Gamma}(x)\longrightarrow\phi(c_{a_1},\dots,c_{a_n}))$ for all $\phi(c_{a_1},\dots,c_{a_n})\in\Gamma$ (modal-type) and $P_{w,\Gamma}\in\sigma'$.\label{typ9}
\end{enumerate}

One can easily see that whenever a tree-like model $\mathfrak{N}$ satisfies axioms \ref{ele1}-\ref{ele5} for any $w\in W_\mathfrak{M}$, then $\mathfrak{M}$ can be elementary embedded in $\mathfrak{N}$.
Furthermore, if $\mathfrak{N}$ also satisfies axioms \ref{typ6}-\ref{typ9}, it realizes all types of $\mathfrak{M}$.

 Let $T'$ be a $\sigma'$-theory containing all of the above sentences for each $w\in W$.
Any modal-type and respectively $\exists$-type are finitely satisfiable in $\mathfrak{M}$. Also, 
$T'$ is finitely satisfiable in (a $\sigma'$-expansion of) $(\mathfrak{M},w_0)$. Hence,  by the compactness of $L$, $T'$ is satisfiable. Let $(\mathfrak{N},v)\models T'$. Since $L$ is invariant under bisimulation, we can consider $\mathfrak{N}$ as a model whose frame is a tree with a countable depth. 
 Therefore, $\mathfrak{N}$ is an $L$-elementary extension of $\mathfrak{M}$ and $\mathfrak{N}$ realizes all types of  $\mathfrak{M}$.
\end{proof}

\begin{prop}\label{sat}
Let $L$ be an abstract compact logic extending FML with the Tarski union property. Suppose that $L$ is invariant under bisimulations. Then every tree-like model $\mathfrak{M}$ whose frame has a countable depth has a modally-saturated  $L$-elementary extension.
\end{prop}
\begin{proof}
We build an $L$-elementary chain $\langle\mathfrak{N}_i: i\in\mathbb{N}\rangle$ as follows.
Let $\mathfrak{N}_0=\mathfrak{M}$. For all $i\geq 1$ let $\mathfrak{N}_{i+1}$ be a tree-like $L$-elementary extension of $\mathfrak{N}_i$ that realizes every type of $\mathfrak{N}_i$. By Proposition \ref{real} and bisimulation invariance this model exists. Let $\mathfrak{N}=\bigcup_{i\in\mathbb{N}}\mathfrak{N}_i$. Since $L$ has TUP, $\mathfrak{N}$ is an $L$-elementary extension of each $\mathfrak{N}_i$. On the other hand, $\mathfrak{N}$ is modally-saturated, since if $\Gamma$ is any type of $(\mathfrak{M},w)$, for some $w\in W_\mathfrak{M}$ there is some $i\in\mathbb{N}$ such that $\Gamma$ is a type of $\mathfrak{N}_i$ and so realized in $\mathfrak{N}_{i+1}$. 
\end{proof}

By the same argument used in \cite{otto:lind08} (Proposition 14) the following proposition is established:

\begin{prop}\label{po}
If $L$ is a compact logic extending FML and there is a sentence $\phi\in L$ not equivalent to any $\xi$ in FML, then there are two models $(\mathfrak{M},w)$ and $(\mathfrak{N},v)$ such that $Th(\mathfrak{M},w)=Th(\mathfrak{N},v)$  but they disagree on $\phi$.
\end{prop}

\begin{theo}\label{lind-tarski}
Any compact and bisimulation invariant logic $L\supseteq \text{FML}$ is equivalent to FML if and only if it satisfies the TUP. 
\end{theo}
\begin{proof}
Assume that there is a sentence $\phi$ in $L$ which is not equivalent to any formula of FML. By Proposition \ref{po}, there are two models $(\mathfrak{M},w)$ and $(\mathfrak{N},v)$  such that $Th(\mathfrak{M},w)=Th(\mathfrak{N},v)$  while $(\mathfrak{M},w)\models\phi$ and $(\mathfrak{N},v)\models\neg\phi$. By invariance under bisimulation, we may assume that $(\mathfrak{M},w)$ and $(\mathfrak{N},v)$ are two tree-like models where their frames have a countable depth. According to Proposition \ref{sat}, there are two saturated models $(\mathfrak{M}^*,w^*)$ and $(\mathfrak{N}^*,v^*)$ such that $Th_{L}(\mathfrak{M},w)=Th_{L}(\mathfrak{M}^*,w^*)$ and $Th_{L}(\mathfrak{N},v)=Th_{L}(\mathfrak{N}^*,v^*)$. It follows that $Th(\mathfrak{M}^*,w)=Th(\mathfrak{N}^*,v)$. Now by the Hennessy-Milner theorem $(\mathfrak{M}^*,w)\rightleftarrows(\mathfrak{N}^*,v)$. Therefore, $(\mathfrak{M}^*,w^*)\models\phi$ and $(\mathfrak{N}^*,v^*)\models\neg\phi$, contradicting invariance under bisimulations.
\end{proof}

\section{Future Research}

This paper is set out with the aim of extending two theorems with propositional nature to first-order modal logic. 

Theorem \ref{gt} provides a generalization of the Goldblatt-Thomason theorem and unravels the expressive power of first-order modal logic in axiomatizing  much broader elementary classes of frames. This result opens a new line of research in finding properties which guarantee that elementary classes of skeletons are first-order modally definable. For further reference in this line of work see  \cite{benthem:frame10} in which the correspondence theory is discussed for some special classes of skeletons.

Furthermore, restriction of Theorem \ref{gt} to constant domain  Kripke models gives raise to the question of finding some alternative results for varying domain Kripke models.

The second part of the present paper is devoted to proving some version of Lindstr\"om characterization theorem  for FML.

There are many ways of generalizing both Theorems \ref{lind} and \ref{lind-tarski}. 
We suggest to find a version of Theorem \ref{lind} which is true for varying domain Kripke models. On the other hand, one may substitute the semantics of first-order modal logic by neighborhood semantics and study the Lindstr\"om theorem in that context.

Our approach in this paper is model theoretic. Since the Goldblatt-Thomason and Lindstr\"om theorems both have coalgebraic forms (see \cite{kurz:golcoal07} and  \cite{venema:coalg10, enq:new14}), it is natural to look for the question of studying coalgebraic adaptation of these theorems.  

\paragraph{Acknowledgements.}
The authors would like to thank Mohsen Khani for carefully reading the paper and giving some useful remarks which helped us to improve the presentation of the paper.

\end{document}